\documentclass[11pt,a4paper,openany,oneside]{article}
\usepackage[a4paper,left=2.5cm,right=2.5cm, top=3cm, bottom=3cm]{geometry}
\usepackage{times}
\usepackage{authblk}
\usepackage[utf8]{inputenc}
\usepackage[english]{babel}
\usepackage{mathrsfs}
\usepackage{stmaryrd}
\usepackage{amssymb,amsmath,amsthm}
\usepackage{hyperref}
\usepackage{graphicx}
\usepackage{color}
\usepackage{enumerate}
\usepackage{multirow}
\usepackage{mathtools}
\usepackage{scalerel}
\usepackage{booktabs}
\usepackage{float}
\usepackage{amsfonts}
\usepackage{cases}
\usepackage{caption}

\usepackage{xcolor,calc}

\theoremstyle{plain}
\newtheorem*{thmA}{Theorem A} 
\newtheorem*{thmB}{Obstruction B}
\newtheorem{thm}{Theorem}[section]
\newtheorem{cor}[thm]{Corollary}
\newtheorem{lem}[thm]{Lemma}

\newtheorem{prop}[thm]{Proposition}

\theoremstyle{definition}

\theoremstyle{remark}
\newtheorem{remark}{Remark}

\definecolor{myred}{rgb}{0.84,0.07,0.14}

\newcommand{\R}{\mathbb R}
\newcommand{\N}{\mathbb N}

\newcommand{\les}{\lesssim}

\newcommand{\dB}{{\partial B}}

\def\e{{\rm e}}

\newcommand{\cN}{{\mathcal N}}

\newcommand{\cP}{{\mathcal P}}

\newcommand{\bR}{{\overline{R}}}
\newcommand{\bu}{{\overline{u}}}
 \def\dd{\, {\rm d}}
\newcommand{\oo}{{\scriptstyle\mathcal{O}}}

\DeclareFontFamily{U}{mathb}{\hyphenchar\font45}
\DeclareFontShape{U}{mathb}{m}{n}{ <-6> matha5 <6-7> matha6 <7-8>
	mathb7 <8-9> mathb8 <9-10> mathb9 <10-12> mathb10 <12-> mathb12 }{}
\DeclareSymbolFont{mathb}{U}{mathb}{m}{n}

\DeclareMathAccent{\abxring}{0}{mathb}{"38}

\DeclareFontFamily{U}{mathb}{\hyphenchar\font45}
\DeclareFontShape{U}{mathb}{m}{n}{ <-6> matha5 <6-7> matha6 <7-8>
	mathb7 <8-9> mathb8 <9-10> mathb9 <10-12> mathb10 <12-> mathb12 }{}
\DeclareSymbolFont{mathb}{U}{mathb}{m}{n}

\newcommand{\uinfty}{u_\infty}

\newcommand{\tu}{\widetilde u}

\DeclareMathOperator{\rad}{rad}

\DeclareOldFontCommand{\it}{\normalfont\itshape}{\mathit}

\newcommand{\supp}{\text{\rm supp}}

\newcommand{\iii}[1]{{\left\vert\kern-0.25ex\left\vert\kern-0.25ex\left\vert #1 \right\vert\kern-0.25ex\right\vert\kern-0.25ex\right\vert}}

\numberwithin{equation}{section}

\newcommand{\todo}[2][]{}

\begin{document}

\renewcommand{\thefootnote}{\fnsymbol{footnote}}
\renewcommand{\thefootnote}{\fnsymbol{footnote}}
\renewcommand{\thefootnote}{\arabic{footnote}}

\title{Positive radial ground states for nonlinear biharmonic equations:\\ maximum principles and expanding-domain approximation}

\date{\today}
\author{
	Daniele Cassani\thanks{Dipartimento di Scienza e Alta Tecnologia, Università degli Studi dell'Insubria and RISM-Riemann International School of Mathematics, Villa Toeplitz, Via G.B. Vico, 46 - 21100 Varese, Italy. Email: daniele.cassani@uninsubria.it},\,
	Zhisu Liu\thanks{Center for Mathematical Sciences/School of Mathematics and Physics, China University of Geosciences, Wuhan, Hubei, 430074, PR China. Email: liuzhisu@cug.edu.cn},\,
	Giulio Romani\thanks{Dipartimento di Scienze Matematiche, Informatiche e Fisiche, Universit\`{a} degli Studi di Udine, Via delle Scienze 206 - 33100 Udine, Italy Email: giulio.romani@uniud.it},\,
	Antonio Tarsia\thanks{Dipartimento di Matematica, Università di Pisa, Largo B. Pontecorvo, 5 -- 56127 Pisa. E-mail: antonio.tarsia@unipi.it}
}

\maketitle


\abstract{We prove the existence of nonnegative radial ground states for a class of nonlinear biharmonic equations in $\R^N$, covering subcritical and critical nonlinearities. The solutions are obtained as limits of clamped ground states on expanding balls. The Cassani-Tarsia homogeneous maximum principle in \cite{CT} makes the approximating states positive. A variational decomposition on the radial Nehari manifold then excludes every zero sphere of positive radius, so the limiting ground state is positive on $\R^N\setminus\{0\}$. When the fourth order operator factorises into two second-order operators with positive resolvents, the conclusion upgrades to $u>0$ throughout $\R^N$, including at the origin. We also give an explicit counterexample showing that the tempting extension of the homogeneous maximum principle to super-solutions with arbitrary nonhomogeneous boundary data is false, even with an arbitrarily large positive zeroth-order coefficient.}

\section{Introduction}

Maximum principles express a simple but powerful idea: information prescribed by an equation and on the boundary can control the sign and the extreme values of a solution throughout a domain. For second-order elliptic equations this principle is one of the main bridges between local differential information and global qualitative behaviour. It underlies comparison and uniqueness arguments, provides a priori bounds, orders sub- and supersolutions, and is a basic ingredient in monotone iteration and fixed-point constructions; see, for example, \cite{PW,GT}. In the viscosity-solution setting, related comparison principles select the physically relevant solution of fully nonlinear equations and Hamilton--Jacobi--Bellman problems \cite{CIL}. Thus a maximum principle is not merely a sign lemma: it is an organizing tool shared by nonlinear analysis, geometric PDEs, optimal control, and the calculus of variations.

Positivity also connects elliptic PDEs with spectral and operator theory. When the resolvent of an elliptic operator maps nonnegative data to nonnegative functions, it behaves as an order-preserving operator on a Banach lattice. Compactness and strong positivity then bring Perron--Frobenius and Krein--Rutman ideas into the PDE setting, singling out a principal eigenvalue and a positive eigenfunction \cite{KR,BNV}. These facts influence stability, bifurcation from the principal eigenvalue, and the structure of nonlinear solution branches. In evolution problems the analogous question is whether the generated semigroup preserves positivity. Even when instantaneous positivity fails for a fourth-order operator, positivity can reappear asymptotically; this phenomenon is formalized by the theory of eventually positive semigroups and has concrete fourth-order elliptic examples \cite{DGK}. Elliptic maximum principles, positive Green functions, inverse-positive operators, and positive or eventually positive flows are therefore different manifestations of the same order structure.

The distinction between second- and higher-order equations is especially important in applications. Fourth-order operators describe the bending of elastic beams and plates, arise in phase-field models such as the Cahn--Hilliard equation, and enter pattern-selection models of Swift--Hohenberg type \cite{GGS,CH,CrossH}. They also occur in geometric energies and in dispersive wave equations, where a biharmonic term changes the balance between localization and oscillation. In these models the sign of a stationary state may encode a physical density, an amplitude, a deflection relative to an obstacle, or the profile of a coherent structure. Positivity can consequently rule out nodal interfaces, identify the lowest-energy branch, support symmetry and comparison arguments, and constrain the long-time dynamics of the corresponding evolution. At the same time, the absence of a general maximum principle warns that intuition inherited from diffusion may be misleading for plate, phase-field, and mixed-dispersion models.

At the linear level, recovering a positivity preserving property yields inverse positivity, sign information for Green operators, and access to principal spectral data. At the nonlinear level, it can distinguish ground states from excited nodal states and combine variational existence with symmetry or comparison arguments. For singular perturbations and expanding-domain limits, weak convergence of positive approximating solutions gives only nonnegativity. In our radial least-energy setting, contact with zero is instead excluded away from the centre by decomposing the solution into its positivity components.

Positivity for higher-order elliptic equations is a long-standing problem, with roots in Boggio's explicit Green function for balls and in Hadamard's early investigations more than a century ago \cite{Boggio,GGS}. Unlike the second-order case, where positivity follows under broad hypotheses from maximum-principle and Harnack theory, the positivity preserving property (PPP)---that is, inheritance of the sign of the data by the solution---depends sensitively on the operator, the domain, and the boundary conditions. For the clamped biharmonic problem, where both the value of the unknown and its normal derivative vanish at the boundary, positivity may fail even on smooth bounded domains. It is recovered on balls and on certain perturbations of balls, while in general one is led to weaker notions such as local, eventual, or dominant positivity. The issue is naturally encoded by the sign of the Green function, whose positive and negative parts determine whether the inverse operator preserves order; see \cite{GGS,GR}. Recent work also shows that sign behaviour for fundamental solutions of general higher-order operators differs sharply from that of products of second-order operators \cite{GRS}. Because comparison, truncation, and monotone methods are central to nonlinear analysis, this failure helps explain why the qualitative theory of higher-order equations remains less complete than its second-order counterpart.

Very recently, however, a new Harnack type inequality opened the way to overcome the lack of maximum principle for higher-order Dirichlet boundary value problems, when the nonpositive highest-order operator is compensated by a sufficiently large positive second-order operator, see \cite{CT} and its extension in \cite{CPT}. Taking advantage of this result, here we establish nonnegative radial ground states, positive away from the origin, for a large class of biharmonic nonlinear equations in the whole space; in the real-factorization regime we obtain positivity also at the origin. More precisely, let $\gamma,\gamma_0\in\R$ and $f$ be a \textit{nonnegative} superlinear and subcritical or critical nonlinearity (the specific assumptions will be given later) and consider
\begin{equation}\label{eq}
	\Delta^2u-\gamma\Delta u+\gamma_0u=f(u)\quad\ \mbox{in}\ \,\R^N,\ \,N\geq2\,.
\end{equation}
This equation can be interpreted as a biharmonic correction of the Schr\"odinger equation for standing waves with constant potential, which can be retrieved by letting $\gamma\to+\infty$, see e.g. \cite{BN}. Many authors have studied existence and multiplicity results for \eqref{eq}, especially for the case $\gamma=0$ and $\gamma_0$ either constant or an $x$-dependent potential with various assumptions on its behaviour at $\infty$, and $f$ with either subcritical or critical polynomial growth (for $N\geq5$), or exponential growth (for $N=4$), so that an exhaustive list is not possible. We cite, among others, \cite{AdO,Sani_subcritical,Sani_critical,YT,dOM,BN,BCM,BCGJ,MSV,CLZ,M}. The list of works dealing with positivity is much shorter. For critical polynomial nonlinearity with decaying potential in dimension $N\geq5$ the question is addressed in \cite{AdO}. The subcritical polynomial case $f(u)=|u|^{p-1}u$ for $p>1$ has been investigated in \cite{BN,M}. For large values of $\gamma$ depending on $\gamma_0$, Bonheure and Nascimento \cite{BN} rewrite the equation as a cooperative system and obtain positivity and symmetry information. For $\gamma=0$ and $\gamma_0=1$, Ma \cite{M} proved the existence of positive radially symmetric ground states through a Talenti-type comparison argument involving $\sqrt{\Delta^2+I}$. In greater generality, positivity was addressed in \cite{dAG} for distributional solutions of equations of the form $\cP(-\Delta)u=\mu$ in $\R^N$, where $\mu$ is a positive measure. When the roots of $\cP$ satisfy an appropriate sign condition, a positive fundamental-solution representation yields positivity of $u$. Applied to \eqref{eq}, however, this imposes both a dimensional restriction and an algebraic restriction on $\gamma$ and $\gamma_0$; for example, $\gamma=\gamma_0=1$ gives the nonreal roots of $t^2+t+1$.

\vskip0.2truecm

The main aim of this manuscript is to prove positivity of radial ground state solutions for a large class of higher-order nonlinear equations in $\R^N$. With \textit{radial ground state} we mean a weak solution which has the least energy among radial solutions in $H^2(\R^N)$, namely solutions in the subspace of radially symmetric functions $H^2_{\rad}(\R^N)$. Our result applies to both subcritical and critical nonlinearities, in large and small dimensions, and without restrictions on the values of $\gamma\geq0$ and $\gamma_0>0$. We focus on the case in which the biharmonic operator is perturbed by a constant coefficient second-order operator as in \eqref{eq}, but we envision that our method can be applied for a more general class of higher-order operators, potentials, and $x$-dependent nonlinearities.

Following ideas going back to Berestycki and Lions \cite{BL_R}, we approximate \eqref{eq} on balls $B_R$ with clamped boundary conditions. Variational mountain-pass methods give radial ground states $u_R$, whose zero extensions belong to the common space $H^2(\R^N)$. For large $R$, the homogeneous maximum principle of \cite{CT,CPT} gives $u_R>0$ in $B_R$. Compactness then produces a nonnegative radial weak limit $\uinfty$ with least radial energy. Its strict positivity away from the origin follows from the least-energy structure: two distinct radial positivity components would each lie on the Nehari manifold and force the energy to be at least twice the ground-state level. If $\gamma^2\geq4\gamma_0$, the operator factors into two positive second-order resolvents, yielding full positivity without any separate monotonicity argument. This also identifies the exact role of \cite{CT,CPT}: it controls the clamped approximants, but it does not yield a maximum principle for arbitrary nonhomogeneous boundary data. Section \ref{Section_MaxPrSupersol} gives an explicit counterexample to that extension. The subcritical and critical cases require different compactness arguments but share the same variational positivity mechanism.
\vskip0.2truecm

To give a flavour of the results before the detailed assumptions, we state the existence conclusion and the obstruction that fixes the scope of the maximum-principle argument:
\begin{thmA}\label{ThmA}
	Let $\gamma\geq0$ and $\gamma_0>0$, and let $f$ be a positive and continuous nonlinearity. Under suitable growth conditions on $f$ at $\infty$ and at $0$, there exists a nonnegative radial ground state solution, positive on $\R^N\setminus\{0\}$, for \eqref{eq}. If moreover $\gamma^2\geq4\gamma_0$, this solution is positive throughout $\R^N$.
\end{thmA}

\begin{thmB}\label{ThmB}
	There is no data-independent extension of the homogeneous clamped maximum
	principle to boundary inequalities $u\geq0$ and
	$\partial_\nu u\leq0$. Already on a disk, arbitrarily large values of
	$\gamma_0$ admit smooth nonnegative data and a smooth nonnegative solution
	with strict boundary inequalities but with an interior zero circle.
\end{thmB}

\vskip0.4truecm
\textbf{Notation}: $\int$ stands for $\int_{\R^N}$, and the differential $\dd x$ is often omitted too. The ball centered in $0$ with radius $R$ is denoted by $B_R$. 
We denote the normal derivative of the function $u$ by $\partial_\nu u$, the positive part by $u_+:=\max\{u,0\}$, and the tensor of all its $k$-derivatives by $D^ku$ for $k\in\N$, with the convention that $D^0u=u$. The space of the infinitely differentiable functions which are compactly supported in a set $\Omega$ is denoted by $C^\infty_0(\Omega)$. The $L^p$-norm is denoted by $\|\cdot\|_p$, and the conjugate H\"older exponent of $p\geq 1$ is $p':=\frac p{p-1}$. 
The symbol $\les$ is used when an inequality is true up to an omitted structural constant.

\subsection{Main results}

Let us now describe in more detail our results. In the sequel we always suppose $N\geq2$, $\gamma\geq0$, $\gamma_0>0$. Let us start by recalling the different concepts of ``criticality'' for a nonlinearity according to the spacial dimension $N$. Motivated by the Sobolev embedding theorem, {the critical growth in dimension $N\geq5$ is the map $t\mapsto t^{2_*-1}$}, where $2_*:=\frac{2N}{N-4}$ is the critical exponent in the embedding $H^2(\R^N)\hookrightarrow L^p(\R^N)$. On the other hand, in dimension $N=4$ we say that $f$ has critical growth in the sense of Trudinger-Moser if there exists $\alpha_0>0$ such that
\begin{equation}\label{crit_def}
	\lim_{t\to+\infty}\frac{f(t)}{\e^{\alpha t^2}}=\begin{cases}
		0&\quad\mbox{for}\ \ \alpha>\alpha_0,\\
		+\infty&\quad\mbox{for}\ \ \alpha<\alpha_0.\end{cases}
\end{equation}
This notion of criticality is motivated by the well-known Adams' inequality, proved first by Adams in bounded domains and later extended by Ruf and Sani in \cite{RS} for the whole space, and which reads as follows: There exists a constant $C>0$ such that
\begin{equation}\label{Adams_RS}
	\sup_{u\in H^2(\R^4),\|u\|_{H^2(\R^4)}\leq1}\,\int_{\R^4}\left(\e^{\alpha u^2}-1\right)\dd x\begin{cases}
		\leq C&\quad\mbox{if}\ \ \alpha\leq32\pi^2,\\
		=+\infty&\quad\mbox{if}\ \ \alpha>32\pi^2.\end{cases}
\end{equation}
As pointed out in \cite{Sani_critical}, the same behaviour holds if the $H^2$-norm is replaced by the equivalent norm $\|\cdot\|_{H^2,\tau}$ defined by
\begin{equation*}
	\|\cdot\|_{H^2,\tau}^2:=\|\Delta\cdot\|_2^2+2\tau\|\nabla\cdot\|_2^2+\tau^2\|\cdot\|_2^2
\end{equation*}
for an arbitrary constant $\tau>0$. In other words, one has
\begin{equation}\label{Adams_RS-tau}
	\sup_{u\in H^2(\R^4),\|u\|_{H^2,\tau(\R^4)}\leq1}\,\int_{\R^4}\left(\e^{\alpha u^2}-1\right)\dd x\begin{cases}
		\leq C&\quad\mbox{if}\ \ \alpha\leq32\pi^2,\\
		=+\infty&\quad\mbox{if}\ \ \alpha>32\pi^2.\end{cases}
\end{equation}

We split now our investigation according to the growth of the nonlinearity $f$, analysing separately the cases of a subcritical and a critical growth. { We just consider dimensions $N\geq4$, since for small dimensions $N\leq3$ our results hold in a simpler way, taking advantage of the Morrey--Sobolev embedding $H^2(\R^N)\hookrightarrow C^{0,\alpha}(\R^N)$ for suitable $\alpha>0$, and hence also into $L^\infty(\R^N)$. In dimensions $N\geq4$, continuity does not follow from the $H^2$ embedding alone; for the solutions considered below it follows instead from the equation by local elliptic regularity and the standard bootstrap allowed by our growth assumptions.} For remarks about our assumptions, as well as their consequences, see Section \ref{Prel}.

In the following we denote by $F(t):=\int_0^tf(\tau)\dd\tau$ the primitive of $f$ which vanishes at $0$.

\paragraph{The subcritical case}
We assume the following conditions on the nonlinearity $f\in C^1(\R)$:
\begin{enumerate}[($f_1$)]
	\item $f(t)>0$ for $t>0$ and $f(t)\equiv0$ for $t\leq0$;
	\item $f(t)=\oo(t)$ as $t\to0$;
	\item $f$ is subcritical, in the sense that
	\begin{itemize}
		\item if $N\geq 5$: there exists $p\in\left(1,2_*-1\right)$ such that $f(t)\les 1+t^p$ as $t\to+\infty$;
		\item if $N=4$: $\displaystyle{\lim_{t\to+\infty}\frac{f(t)}{\e^{\alpha t^2}}=0}$ for all $\alpha>0$;
	\end{itemize}
	\item there exists $\mu>2$ such that $\mu F(t)\leq tf(t)$ for all $t>0$;
	\item the map $t\mapsto\frac{f(t)}t$ is monotone increasing in $\R^+$.
\end{enumerate}

\noindent In this setting, Theorem \ref{ThmA} reads as follows:

\begin{thm}\label{Thm2}
	Let $\gamma_0>0$ and $\gamma\geq0$, and suppose conditions ($f_1$)-($f_5$) hold. Then there exists a nonnegative radial ground state solution, positive on $\R^N\setminus\{0\}$ for \eqref{eq}. If moreover $\gamma^2\geq4\gamma_0$, the solution is positive throughout $\R^N$.
\end{thm}

\paragraph{The critical case}

Let us start with the case $N\geq5$. Here, we deal with subcritical perturbations of the critical nonlinearity, that is
\begin{equation}\label{NL-C}
	f(t)=t_+^{2_*-1}+g(t)
\end{equation}
with the following assumptions on $g\in C^1(\R^+)$:
\begin{enumerate}
	\item[($f_3'$)] there exist $q\in(1,2_*-1)$ and $\eta>0$ such that
	\begin{itemize}
		\item $g(t)\leq C(1+t^q)$ for all $t>0$ and $g\equiv0$ for all $t<0$;
		\item $g(t)\geq\eta t^{p-1}$ with either $\eta>0$ and $p\in(\overline p(N),2_*)$, or $p\in(2,\overline p(N))$ and $\eta>\eta_0$ for some $\eta_0>0$ large enough. Here we define
		{\begin{equation*}
			\overline p(N)=\begin{cases}
				\frac{2(N-2)}{N-4}&\quad\mbox{for}\ \ N\geq8\ \ \mbox{and}\ \ N=6,\\
				\frac83&\quad\mbox{for}\ \ N=7,\\
				8&\quad\mbox{for}\ \ N=5.\end{cases}
		\end{equation*}}
	\end{itemize}
\end{enumerate}

\begin{thm}\label{Thm-C-5-general}
	Let $\gamma_0>0$, $\gamma\geq0$, and $f$ as in \eqref{NL-C}, with $g$ satisfying ($f_1$), ($f_2$), ($f_3'$), ($f_4$), ($f_5$). Then there exists a nonnegative radial ground state solution, positive on $\R^N\setminus\{0\}$ for \eqref{eq}. If moreover $\gamma^2\geq4\gamma_0$, the solution is positive throughout $\R^N$.
\end{thm}

\vskip0.4truecm

In dimension $N=4$ the critical growth is the exponential $t\mapsto\e^{\alpha t^2}$. Note that if a function has critical growth in the sense of \eqref{crit_def} and satisfies ($f_2$), then for fixed $\alpha>\alpha_0$, $q\geq1$ and for any $\varepsilon>0$ one has
\begin{equation}\label{f-C-above}
	|f(t)|\leq\varepsilon|t|+C(\alpha,q,\varepsilon)|t|^{q-1}(\e^{\alpha t^2}-1),\qquad t\in\R,
\end{equation}
where $C(\alpha,q,\varepsilon)>0$, and consequently,
\begin{equation}\label{F-C-above}
	|F(t)|\leq\varepsilon|t|^2+C(\alpha,q,\varepsilon)|t|^q(\e^{\alpha t^2}-1),\qquad t\in\R.
\end{equation}

\begin{thm}\label{Thm-C-4}
	Let $\gamma_0>0$ and $\gamma\geq0$, and $f$ be a critical nonlinearity as in \eqref{crit_def}. In addition to assumptions ($f_1$), ($f_2$), ($f_4$), ($f_5$), suppose $f$ also satisfies
	\begin{enumerate}
		\item[($f_6$)] there exist $t_0>0$ and $M_0>0$ such that $0<F(t)\leq M_0|f(t)|$ for any $|t|\geq t_0$;
		\item[($f_7$)] there exists $\beta_0>0$ such that $\displaystyle{\lim_{t\to+\infty}\frac{tf(t)}{\e^{\alpha_0t^2}}\geq\beta_0}$.
	\end{enumerate}
	Then there exists a nonnegative radial ground state solution, positive on $\R^N\setminus\{0\}$ for \eqref{eq}. If moreover $\gamma^2\geq4\gamma_0$, the solution is positive throughout $\R^N$.
\end{thm}

\paragraph{The maximum-principle component}
The expanding-ball construction uses the homogeneous clamped maximum
principle of \cite{CPT}. Section \ref{Section_MaxPrSupersol} records its exact
scope, disproves the unrestricted nonhomogeneous extension, and replaces that
invalid step by a radial least-energy component argument.
\vskip0.2truecm
\textbf{Overview}: Section \ref{Prel} gathers the variational and
functional-analytic preliminaries. Section \ref{Section_MaxPrSupersol} recalls
the homogeneous maximum principle, disproves its unrestricted nonhomogeneous
extension, and establishes the radial component lemma used for strict
positivity away from the origin. The subcritical case is treated in Section
\ref{Section_Subcritical}; Sections \ref{Section_crit_largedim} and
\ref{Section_crit_N4} address critical polynomial and critical exponential
growth, respectively.

\section{Preliminaries and consequences of our assumptions}\label{Prel}

The map
$$u\mapsto\iii{u}_{\gamma,\gamma_0}:=\left(\int\left(|\Delta u|^2+\gamma|\nabla u|^2+\gamma_0 u^2\right)\right)^\frac12$$
is a norm on $H^2(\R^N)$ which is equivalent to the standard one, provided $\gamma>-2\sqrt\gamma_0$, see \cite[Lemma 2.2]{BN}. We denote by $C_{eq}$ the best constant in the equivalence, namely such that $\|u\|_{H^2(\R^N)}^2\leq C_{eq}\iii u^2_{\gamma,\gamma_0}$ for all $u\in H^2(\R^N)$. Of course, $C_{eq}$ depends on the parameters $\gamma,\gamma_0$. The same holds for
$$u\mapsto\iii{u}_{\gamma,\gamma_0;R}:=\left(\int_{B_R}\left(|\Delta u|^2+\gamma|\nabla u|^2+\gamma_0 u^2\right)\right)^\frac12,$$
for any $R>0$, see \cite[Lemma 2.1]{BN}. Henceforth we omit the indices $\gamma$ and $\gamma_0$ from the norm by sake of a shorter notation.

\paragraph{Remarks on our assumptions}
	\begin{enumerate}
		\item Since we are interested in \textit{positive} solution, supposing $f\equiv0$ on $\R^-$ as in ($f_1$) is not restricting; ($f_2$) and ($f_3$) are standard growth condition of $f$ a $0$ and at $\infty$, respectively, which classify $f$ as superlinear at $0$ and subcritical or critical at $\infty$; ($f_4$) is the Ambrosetti-Rabinowitz condition, which implies that for $t>0$ one has $F(t)\geq c_1t^\mu-c_2$ for some $c_1,c_2>0$; ($f_5$) is a standard assumption to have a good Nehari geometry.
		\item In the critical cases, we further need the additional assumptions ($f_3'$) when $N\geq5$, and ($f_6$)-($f_7$) when $N=4$. In the former setting, ($f_3'$) prescribes that we perturb the critical nonlinearity by a subcritical term in order to retrieve compactness by a careful analysis of the ground state level (see Lemma \ref{comp} below); in the latter, ($f_6$)-($f_7$) are imposed to be able to profit from the existence results of \cite{Sani_critical}. In particular, from them one obtains a sharp upper bound of the mountain pass level, see \eqref{c_mp_Sani} below, and get compactness. We point out that imposing the de Figuereido-Miyagaki-Ruf condition ($f_7$) in the spirit of \cite{dFMR}, allows to avoid a global growth condition on $f$ from below, as frequently used in the literature, see e.g. \cite{dOM,MSV}, but not of practical verification. Moreover, the constant $\beta_0$ in ($f_7$) is not prescribed to be large as originally in \cite{dFMR}.
		\item As a consequence of ($f_2$)-($f_3$) one infers that if $N\geq5$ for all $\varepsilon>0$ there exists $C_\varepsilon>0$ such that
		\begin{equation}\label{f_above_N>=5}
			f(t)\leq\varepsilon t+C_\varepsilon t^p,\qquad F(t)\leq\varepsilon t^2+C_\varepsilon t^{p+1}\qquad\mbox{for all}\ \,t>0\,,
		\end{equation}
		and, if $N=4$, for all $\varepsilon>0$, $q>0$ and $\alpha>0$ there exists $C=C(\varepsilon,q,\alpha)$ such that
		\begin{equation}\label{f_above_N=4}
			f(t)\leq\varepsilon t+C t^q\left(\e^{\alpha t^2}-1\right),\quad F(t)\leq\varepsilon t^2+Ct^q\left(\e^{\alpha t^2}-1\right)\quad\mbox{for all}\ \,t>0\,.
		\end{equation}
	\end{enumerate}
\begin{remark}
	We prove the existence of a nonnegative radial solution of \eqref{eq} which
	is a ground state among radial solutions and is positive at every point of
	$\R^N\setminus\{0\}$. We do not claim that every solution has this sign.
	The expanding-ball construction supplies nonnegativity of the limit, while
	the exclusion of zero spheres uses essentially its least-energy character.
\end{remark}
\vskip0.2truecm
Before we proceed, we recall a useful estimate in the case $N=4$.
\begin{lem}[{\cite[Lemma 2.1]{Y}}]\label{estimate_Sani}
	Let $\alpha>0$ and $r>1$. Then
	\begin{equation}\label{estimate_Sani_new}
		(\e^{\alpha t^2}-1)^r\leq (\e^{\alpha rt^2}-1)\qquad\mbox{for all}\ \,t>0.
	\end{equation}
\end{lem}

\section{What the maximum principle provides, and an obstruction for nonhomogeneous data}\label{Section_MaxPrSupersol}

The approximation argument uses the established positivity result for the
homogeneous clamped problem.  We recall it in the form needed below.

\begin{lem}[\cite{CPT}]\label{Pos_CT}
Let $\Omega\subset\R^N$, $N\geq2$, be a bounded connected smooth domain
satisfying a uniform interior sphere condition.  Let $f\in H^k(\Omega)$,
$k>\frac N2-4$, satisfy $f\geq0$ and $f\not\equiv0$.  If $u$ solves
\[
\begin{cases}
\Delta^2u-\gamma\Delta u+\gamma_0u=f&\text{in }\Omega,\\
u=\partial_\nu u=0&\text{on }\partial\Omega,
\end{cases}
\]
with $\gamma,\gamma_0\geq0$, then there is
$\mu_0=\mu_0(\Omega,N)$, independent of $f$, such that
$\gamma+\gamma_0>\mu_0$ implies $u>0$ in $\Omega$.
\end{lem}

It is tempting to replace the homogeneous boundary conditions in Lemma
\ref{Pos_CT} by $u\geq0$ and $\partial_\nu u\leq0$.  The following example
shows that no result of this form can hold for arbitrary nonnegative data,
even when the zeroth-order coefficient is arbitrarily large.

\begin{prop}[Obstruction to a nonhomogeneous supersolution principle]
\label{prop:counterexample-mp}
Let $\Omega=B_2\subset\R^2$ and fix any $\gamma\geq0$. There exists
$\gamma_* = \gamma_*(\gamma)>0$ such that, for every
$\gamma_0\geq\gamma_*$, one can find
$u,f\in C^\infty(\overline\Omega)$ satisfying
\[
\Delta^2u-\gamma\Delta u+\gamma_0u=f\geq0\quad\text{in }\Omega,
\qquad u>0,\quad\partial_\nu u<0\quad\text{on }\partial\Omega,
\]
while $u\geq0$ in $\Omega$ and $u$ vanishes on the interior circle
$\{|x|=1\}$.  In addition, $f$ vanishes on that circle and is not identically
zero.
\end{prop}

\begin{proof}
Put $r=|x|$ and
\[
u(r)=(r^2-1)^6\mathrm e^{-3r^2}.
\]
This is a smooth nonnegative radial function and its zero set in $B_2$ is
exactly $\{r=1\}$.  If $U(s)=(s-1)^6\mathrm e^{-3s}$, then in dimension two
\[
\Delta^2 u(r)=16r^4U^{(4)}(r^2)+64r^2U^{(3)}(r^2)+32U''(r^2).
\]
Writing $t=s-1$, the relevant terms have the expansions
\[
\Delta^2u=5760\mathrm e^{-3}t^2+O(t^3),\qquad
\Delta u=120\mathrm e^{-3}t^4+O(t^5),\qquad
u=\mathrm e^{-3}t^6+O(t^7).
\]
Therefore, for every fixed $\gamma\geq0$,
$\Delta^2u-\gamma\Delta u>0$ in a punctured neighbourhood of $r=1$.
On the complement of a sufficiently small annulus about $r=1$, the quotient
$(\Delta^2u-\gamma\Delta u)/u$ is continuous and bounded below. Hence there
is $\gamma_*(\gamma)>0$ such that
$f:=\Delta^2u-\gamma\Delta u+\gamma_0u\geq0$ for every
$\gamma_0\geq\gamma_*(\gamma)$. Since $u$ has a zero of order six at
$r=1$, while $\Delta u$ and $\Delta^2u$ vanish there as well, also $f=0$ on
that circle. Finally,
$u(2)>0$ and
\[
u'(2)=4U'(4)=-2916\mathrm e^{-12}<0.
\]
Thus all asserted properties hold for every fixed $\gamma\geq0$.
\end{proof}

Proposition \ref{prop:counterexample-mp} explains why an
Harnack absorption does not work in the nonhomogeneous setting. The difficulty is structural, not just technical. We therefore cannot use a maximum principle with
nonhomogeneous boundary data.  Instead, strict positivity away from the radial
centre follows from the least-energy structure.

\begin{lem}[Exclusion of zero spheres]\label{lem:no-zero-spheres}
Assume $f\in C^1(\R)$, $f(0)=0$, and that the radial Nehari manifold
\[
\cN=\{v\in H^2_{\rad}(\R^N)\setminus\{0\}:J'(v)v=0\}
\]
is a natural constraint with positive level
$c_{\rad}=\inf_{\cN}J>0$.  Let $u\in C^4(\R^N)\cap H^2_{\rad}(\R^N)$ be a
nontrivial, nonnegative critical point satisfying $J(u)=c_{\rad}$.  Then
\[
u(x)>0\qquad\text{for every }x\ne0.
\]
\end{lem}

\begin{proof}
Write $u(x)=w(|x|)$ and consider the open set
$P=\{r>0:w(r)>0\}$.  If $P$ has two distinct connected components $I_1$ and
$I_2$, set $u_i(x)=u(x)\mathbf1_{\{|x|\in I_i\}}$.  At every finite endpoint
of $I_i$, nonnegativity and $C^1$ regularity give $w=w'=0$; hence the zero
extension $u_i$ belongs to $H^2_{\rad}(\R^N)$.  Testing $J'(u)=0$ with $u_i$
gives $J'(u_i)u_i=0$, so $u_i\in\cN$.  Since the supports are disjoint up to
sets of measure zero,
\[
J(u)=\sum_I J(u_I)\geq J(u_1)+J(u_2)\geq2c_{\rad},
\]
contrary to $J(u)=c_{\rad}$. Thus $P$ has exactly one component.

If this component were different from $(0,\infty)$, then $w$ would vanish on
a nonempty open interval contained in $(0,\infty)$.  The radial equation
associated with $\Delta^2u-\gamma\Delta u+\gamma_0u=f(u)$ is a fourth-order
ordinary differential equation with locally Lipschitz right-hand side on every
interval separated from $r=0$. At an endpoint adjacent to the zero interval,
$w,w',w'',w'''$ all vanish. Uniqueness for the corresponding first-order
system would force $w\equiv0$, a contradiction.  Therefore
$P=(0,\infty)$.
\end{proof}

\begin{remark}\label{rmk:origin}
Lemma \ref{lem:no-zero-spheres} removes every possible zero sphere of positive
radius.  It does not, without an additional monotonicity or comparison input,
exclude the isolated alternative $u(0)=0<u(r)$ for $r>0$. Accordingly, the
results below assert positivity on $\R^N\setminus\{0\}$; full pointwise
positivity follows whenever $u(0)>0$ is known independently. This distinction
is immaterial for positivity almost everywhere, but is essential for a correct
pointwise statement.
\end{remark}

The isolated alternative in Remark \ref{rmk:origin} can be excluded when the
linear operator has two real positive second-order factors.  This familiar
factorization is also the basis of the cooperative-system approach in
\cite{BN}; we record the short resolvent argument because it applies directly
to the ground states constructed here.

\begin{prop}[Full positivity in the real-factorization regime]
\label{prop:factorization-positive}
Let $\gamma_0>0$, $\gamma\geq0$, and $\gamma^2\geq4\gamma_0$.  Suppose that
$u\in H^2(\R^N)$ is a nontrivial distributional solution of
\[
\Delta^2u-\gamma\Delta u+\gamma_0u=h\quad\ \text{in }\ \,\R^N,
\]
where $h\geq0$, $h\not\equiv0$, and the equation is well defined in
$H^{-2}(\R^N)$. Then $u$ has a strictly positive resolvent representative.
In particular, if $u$ is continuous, then
\[
u(x)>0\qquad\text{for every }x\in\R^N.
\]
In particular, this conclusion applies to every nontrivial nonnegative
solution of \eqref{eq} under the hypotheses of the main theorems.
\end{prop}

\begin{proof}
Set
\[
\lambda_\pm=\frac{\gamma\pm\sqrt{\gamma^2-4\gamma_0}}2.
\]
Since $\lambda_++\lambda_-=\gamma$ and
$\lambda_+\lambda_-=\gamma_0>0$, both numbers are positive and
\[
\Delta^2-\gamma\Delta+\gamma_0
=(-\Delta+\lambda_+)(-\Delta+\lambda_-).
\]
Let $G_\lambda$ denote the Bessel kernel, namely the fundamental solution of
$-\Delta+\lambda$ in $\R^N$.  It is strictly positive.  The Fourier-multiplier
representation of the unique $H^2$ solution gives
\[
u=G_{\lambda_-}*G_{\lambda_+}*h.
\]
Equivalently, $v:=G_{\lambda_+}*h$ solves
$(-\Delta+\lambda_+)v=h$ and $u=G_{\lambda_-}*v$.  Positivity of the Bessel
kernels and $0\not\equiv h\geq0$ imply first $v>0$ and then $u>0$ everywhere.
The equality case $\gamma^2=4\gamma_0$ is included: then
$\lambda_+=\lambda_-=\gamma/2>0$ and the same convolution argument uses the
square of a single positive resolvent.

For the nonlinear problem, take $h=f(u)$. The constructed ground state is
nonnegative and positive for $x\ne0$ by Lemma \ref{lem:no-zero-spheres}; by
($f_1$), $f(u)\geq0$ and $f(u)\not\equiv0$. When $N\leq3$, its continuity is
immediate from the Morrey--Sobolev embedding of $H^2(\R^N)$. When $N\geq4$,
the $H^2$ embedding alone is not sufficient, but local elliptic regularity
and bootstrap applied to \eqref{eq}, under the stated subcritical or critical
growth assumptions, give a continuous representative (in fact, the higher
regularity used elsewhere in the paper). The positive resolvent
representation is therefore pointwise and yields, in particular, $u(0)>0$.
\end{proof}

\section{The subcritical case: Proof of Theorem \ref{Thm2}}\label{Section_Subcritical}\label{Proof-S}

Let $R>0$ and consider the biharmonic Dirichlet problem
\begin{equation}\label{eq_R}
	\begin{cases}
		\Delta^2u-\gamma\Delta u+\gamma_0 u=f(u)&\;\mbox{in}\;B_R,\\
		u=\partial_\nu u=0&\;\mbox{on}\;\dB_R,
	\end{cases}
\end{equation}
with $\gamma\geq0$ and $\gamma_0>0$. The advantage of taking Dirichlet boundary conditions on $\dB_R$ is that one can extend by $0$ any function in $H^2_0(B_R)$ to a function in $H^2(\R^N)$. In the sequel, as well as in the following sections, the $0$-extension of a function $u$ is denoted by $\tu$. Since the problem has a radial symmetry, by means of the principle of symmetric criticality, see \cite[Theorem 1.28]{Willem}, one may consider the problem in $H^2_{0,\rad}(B_R)$, which is the subspace of radially symmetric functions in $H^2_0(B_R)$.

Since $\iii{u}_R$ is a norm in $H^2_0(B_R)$, the standard mountain-pass argument under assumptions ($f_1$)-($f_4$) gives a solution $u_R\in H^2_{0,\rad}(B_R)$, which is also a radial ground state if ($f_5$) holds; see, for instance, \cite[Section 5]{dOM}. More precisely, associated to the variational equation \eqref{eq_R} is the functional $J:H^2_0(B_R)\to\R$ defined by
$$J_R(u)=\frac12\iii{u}_R^2-\int_{B_R}F(u),$$
in the sense that critical points of $J_R$ correspond to weak solutions of \eqref{eq_R}.
It is easy to verify from our assumptions on $f$ that $J\in C^1$. Moreover, by ($f_4$) one has $J(tu)\to-\infty$ for all $u\in H^2_0(B_R)$ fixed as $t\to+\infty$, and assumption ($f_2$) yields $J_R(u)>\delta_R>0$ on $S_\rho^R:=\{u\in H^2_{0,\rad}(B_R)\,|\,\iii{u}_R=\rho_R\}$ for some $\delta_R,\rho_R>0$. Hence the mountain-pass level
\begin{equation}\label{MPlevel}
	c_R:=\inf_{\gamma\in\Gamma_R}\max_{u\in\gamma([0,1])}J_R(u),
\end{equation}
is well-defined, where
$$\Gamma_R:=\left\{\gamma\in C\left([0,1],H^2_{0,\rad}(B_R)\right)\,|\,\gamma(0)=0,\gamma(1)=v_0^R\right\}$$
and $v_0^R\in H^2_{0,\rad}(B_R)$ such that $J_R(v_0^R)<0$. Before going on, in Lemma \ref{MPunif}, we retrace this argument, showing that the point $v_0^R$ and the constants $\delta_R$ and $\rho_R$ can be chosen independently of $R$.

\begin{lem}\label{MPunif}
	There exist $\delta,\rho>0$ and $v_0\in H^2_{0,\rad}(B_1)$ such that $J_R(u)>\delta>0$ for all $u\in S_\rho:=\{u\in H^2_0(B_R)\,|\,\iii{u}_R=\rho\}$, and $J_R(v_0)<0$ for all $R\geq 1$.
\end{lem}
\begin{proof}
	Take $e_0\in C^\infty_{0,\rad}(B_1)$, then, using ($f_4$) for $t>0$ one has
	\begin{equation*}
		\begin{split}
			J_R(te_0)&\leq\frac{t^2}2\iii{e_0}_R^2-t^\mu\int_{B_R}e_0^\mu=\frac{t^2}2\iii{e_0}_1^2-t^\mu\int_{B_1}e_0^\mu\to-\infty
		\end{split}
	\end{equation*}
	as $t\to+\infty$ uniformly for $R\geq1$. If we define then $v_0:=te_0\in C^\infty_{0,\rad}(B_1)\subset H^2_{0,\rad}(B_R)$ for all $R\geq1$ and take $t$ large enough, we get the desired element.
	\vskip0.2truecm
	Let $N\geq5$. By \eqref{f_above_N>=5} one gets
	{\begin{equation*}
		\begin{split}
			J_R(u)&\geq\frac{\iii{u}_R^2}2-\varepsilon\int_{B_R}|u|^2-C_\varepsilon\int_{B_R}|u|^{p+1}\geq\left(\frac12-\varepsilon\right)\frac{\iii{\tu}^2}2-C_\varepsilon C_p\|\tu\|_{H^2(\R^N)}^{p+1}\\
			&\geq\left(\frac12-\varepsilon\right)\frac{\iii{\tu}^2}2-C_\varepsilon C_pC_{eq}^\frac{p+1}2\iii{\tu}^{p+1},
		\end{split}
	\end{equation*}}
	where $C_p$ is the best constant in the critical embedding $H^2(\R^N)\hookrightarrow L^{p+1}(\R^N)$. Hence, choosing $\rho>0$ small enough, on $S_\rho$ one gets $J_R(u)>\delta>0$ with $\delta$ independent of $R$.
	
	Let now $N=4$, then by \eqref{f_above_N=4} for all $u\in H^2_{0,\rad}(B_R)$ and $\varepsilon>0$, $q>2$, $\alpha>0$, and $t>1$, one infers
	\begin{equation*}
		\begin{split}
			\int_{B_R}F(u)&\leq\varepsilon\|u\|_2^2+C(\varepsilon,q,\alpha)\int_{B_R}|u|^q\left(\e^{\alpha u^2}-1\right)\\
			&\leq\varepsilon\|u\|_2^2+C(\varepsilon,q,\alpha)\|\tu\|_{qt'}^q\left(\int\left(\e^{\alpha t\tu^2}-1\right)\right)
		\end{split}
	\end{equation*}
	by Lemma \ref{estimate_Sani}. Choosing $\alpha$ and $t$ such that $\alpha t=32\pi^2$, then by Adams' inequality \eqref{Adams_RS} we get
	\begin{equation*}
		\int_{B_R}F(u)\leq\varepsilon\|u\|_2^2+C(\varepsilon,q)\|\tu\|_{qt'}^q\leq C\varepsilon\iii{\tu}^2+C(\varepsilon,q)\iii{\tu}^q,
	\end{equation*}
	which yields
	\begin{equation*}
		J_R(u)=\frac12\iii{u}_R^2-\int_{B_R}F(u)\geq\left(\frac12-C\varepsilon\right)\iii{\tu}^2-C(\varepsilon,q)\iii{\tu}^q.
	\end{equation*}
	For a fixed $q>2$ it is sufficient to choose $\varepsilon=\frac1{4C}$ and $\rho$ small enough, so that again for $u\in S_{\rho}$, $J_R(u)>\delta>0$ with $\delta$ independent of $R$.
\end{proof}

\begin{cor}\label{MPlevel_bdd}
	The mountain-pass level map $R\mapsto c_R$ is decreasing and $0<\delta<c_R\leq c_1$ for all $R\geq1$.
\end{cor}
\begin{proof}
	Note that if $R_1<R_2$ then $H^2_{0,\rad}(B_{R_1})\subset H^2_{0,\rad}(B_{R_2})$ by the $0$-extension. Hence $\Gamma_{R_1}\subset\Gamma_{R_2}$ and therefore $c_{R_1}\geq c_{R_2}$. The fact that $c_R>\delta>0$ follows directly from Lemma \ref{MPunif}.
\end{proof}

We are now going to obtain a radial solution of our biharmonic Schr\"odinger problem \eqref{eq} on $\R^N$ as the limit of the sequence of the radially symmetric mountain-pass solutions $(\tu_R)_R$ obtained before. In this way, if one proves positivity for the $u_R$ for $R$ large, the limit point will not change sign.

\begin{prop}\label{exists_uinfty}
	The (extended) family of mountain-pass solutions $(\tu_R)_R$ is uniformly bounded in $H^2_{\rad}(\R^N)$. Moreover, there exists $u_\infty\in H^2_{\rad}(\R^N)$ such that, up to a subsequence, $\tu_R\to u_\infty$ weakly in $H^2(\R^N)$ and strongly in $L^q(\R^N)$ for all $q\in(2,2_*)$, and $\uinfty$ is a weak solution of \eqref{eq}.
\end{prop}
\begin{proof}
	Since we know that, for any $R\geq1$, $u_R$ is a radial weak solution of \eqref{eq_R}, then $J'_R(u_R)u_R=0$, that is $\iii{u_R}_R^2=\int_{B_R}f(u_R)u_R$. Hence, by ($f_4$)
	\begin{equation*}
		c_R=J_R(u_R)=\frac12\iii{u_R}_R^2-\int_{B_R}F(u_R)\geq\frac12\iii{u_R}_R^2-\frac1\mu\int_{B_R}f(u_R)u_R=\left(\frac12-\frac1\mu\right)\iii{u_R}_R^2.
	\end{equation*}
	By Corollary \ref{MPlevel_bdd}, we then infer
	\begin{equation}\label{iii_unifbdd}
		\iii{\tu_R}^2=\iii{u_R}_R^2\leq\frac{2\mu}{\mu-2}c_R\leq\frac{2\mu}{\mu-2}c_1.
	\end{equation}
	Therefore there exists $\uinfty\in H^2_{\rad}(\R^N)$ and a subsequence $(u_k)_k:=(u_{R_k})_{R_k}$ with $R_k\nearrow+\infty$ as $k\to+\infty$, such that $\tu_k\rightharpoonup\uinfty$ in $H^2_{\rad}(\R^N)$ and by compact embedding, see \cite{Lions}, $\tu_k\to\uinfty$ in $L^q(\R^N)$ for all $q\in(2,2_*)$ if $N\geq5$ or for all $q\in(2,+\infty)$ if $N=4$. Let now $\varphi\in C^\infty_0(\R^N)$, for $k$ large enough one has $B_{R_K}(0)\supset\supp\,\varphi$ and hence
	\begin{equation}\label{conv_1-S}
		\int_{B_{R_k}}\left(\Delta u_k\Delta\varphi+\gamma\nabla u_k\nabla\varphi+\gamma_0 u_k\varphi\right)\to\int\left(\Delta\uinfty\Delta\varphi+\gamma\nabla\uinfty\nabla\varphi+\gamma_0\uinfty\varphi\right)
	\end{equation}
	by weak convergence. It remains to prove that $\int_{B_{R_k}}f(u_k)\varphi\to\int f(\uinfty)\varphi$, and we distinguish again two cases according to the dimension. If $N\geq5$, it is easy to see that
	$$\int_{\supp\,\varphi}f(u_k)|\varphi|\les\int_{\supp\,\varphi}\left(1+|\tu_k|^p\right)\leq C,$$
	being $(\tu_k)_k$ bounded in $L^p(\R^N)$, since $p<2_*$. By continuity of $f\geq0$, by Fatou's lemma, we consequently have
	$$\int_{\supp\,\varphi}f(\uinfty)|\varphi|\leq\liminf_{k\to+\infty}\int_{\supp\,\varphi}f(\tu_k)|\varphi|\leq C.$$
	In order to apply \cite[Lemma 2.1]{dFMR} and infer $\int_{B_{R_k}}f(u_k)\varphi\to\int f(\uinfty)\varphi$, one should verify that $\left(f(\tu_k)\tu_k\varphi\right)_k$ is uniformly bounded in $L^1(\supp\,\varphi)$:
	$$\int_{\supp\,\varphi}\left|f(\tu_k)\tu_k\varphi\right|\les\int_{\supp\,\varphi}\left(1+|\tu_k|^p\right)|\tu_k|\les 1+\int_{\supp\,\varphi}|u_k|^{p+1}\leq C$$
	since $p<2_*-1$ by ($f_3$).
	
	Let now $N=4$; by \eqref{f_above_N=4} with $\varepsilon=1$ and $q\geq1$, Lemma \ref{estimate_Sani}, and \eqref{iii_unifbdd}, we infer
	\begin{equation*}
		\begin{split}
			\int_{\supp\,\varphi}f(\tu_k)|\varphi|&\leq\int_{\supp\,\varphi}|\tu_k||\varphi|+C\int_{\supp\,\varphi}|\tu_k|^q\left(\e^{\alpha\tu_k^2}-1\right)|\varphi|\\
			&\les\|\tu_k\|_2+\left(\int|\tu_k|^{qr'}\right)^\frac1{r'}\left(\int\left(\e^{\alpha\tu_k^2}-1\right)^r\right)^\frac1r\\
			&\les\sqrt{\frac{2\mu c_1}{(\mu-2)\gamma_0}}+C(q,r)\iii{\tu_k}^q\left(\int\left(\e^{\alpha r\tu_k^2}-1\right)\right)^\frac1r
		\end{split}
	\end{equation*}
	with $r>1$. By choosing now $\alpha<\tfrac{16\pi^2(\mu-2)}{r\mu C_{eq}c_1}$ we get
	\begin{equation*}
		\alpha r\|\tu_k\|_{H^2(\R^4)}^2\leq\alpha rC_{eq}\iii{\tu_k}^2\leq\alpha rC_{eq}\frac{2\mu}{\mu-2}c_1\leq32\pi^2,
	\end{equation*}
	and therefore by Adams' inequality and \eqref{iii_unifbdd} we get
	\begin{equation*}
		\int_{\supp\,\varphi}f(\tu_k)|\varphi|\leq C.
	\end{equation*}
	Again by Fatou's lemma we have $\int_{\supp\,\varphi}f(\uinfty)|\varphi|<+\infty$ and with similar steps we also get                                                                                                                                                                                                                                                                                                                                                                                                                                                                                                            $\int_{\supp\,\varphi}\left|f(\tu_k)\tu_k\varphi\right|\leq C$. The proof is completed applying again \cite[Lemma 2.1]{dFMR}.
\end{proof}

\begin{lem}\label{lem_conv}
	We have
	$$\int f(\tu_k)\tu_k\to\int f(\uinfty)\uinfty\quad\ \,   \mbox{as}\ \,\,k\to+\infty.$$
\end{lem}
\begin{proof}
	Once more, we need to distinguish between the fourth-dimensional case and the higher-dimensional case. Let us begin with $N\geq 5$. We split
	\begin{equation*}
		\left|\int f(\tu_k)\tu_k-\int f(\uinfty)\uinfty\right|\leq\int f(\tu_k)|\tu_k-\uinfty|+\left|\int\left(f(\tu_k)-f(\uinfty)\right)\uinfty\right|=:T_1+T_2,
	\end{equation*}
	and estimate the two terms separately. For any $\varepsilon>0$ using ($f_3$) we obtain
	\begin{equation*}
		T_1\leq\int\left(\varepsilon|\tu_k|+C_\varepsilon|\tu_k|^p\right)|\tu_k-\uinfty|\leq\varepsilon\|\tu_k\|_2\|\tu_k-\uinfty\|_2+C_\varepsilon\|\tu_k\|_{pq}^p\|\tu_k-\uinfty\|_{q'}
	\end{equation*}
	for $q>1$. Since $(\tu_k)_k$ is bounded in $H^2(\R^N)$ by Proposition \ref{exists_uinfty} and $\uinfty\in H^2(\R^N)$, the first term in $T_1$ is bounded by $C\varepsilon$ for some suitable constant $C$ independent of $\varepsilon$. Moreover, choosing $q=\frac{p+1}p$ we get $\|\tu_k\|_{pq}\leq C\|\tu_k\|_{H^2(\R^N)}\leq C$ and $q'=p+1\in(2,2_*)$, so by the compact embedding the second term vanishes in the limit. Let $\varepsilon>0$ and $\bR>0$ to be fixed later and split
	\begin{equation}\label{T_2}
		T_2\leq\left|\int_{B_\bR}\left(f(\tu_k)-f(\uinfty)\right)\uinfty\right|+\left|\int_{\R^N\setminus B_\bR}\left(f(\tu_k)-f(\uinfty)\right)\uinfty\right|.
	\end{equation}
	By elliptic regularity and continuity of $f$ we deduce that $\uinfty$ is smooth and $\|\uinfty\|_{L^\infty(B_\bR)}\leq C_\bR$. Moreover, in the proof of Proposition \ref{exists_uinfty} we have already shown that $f(\tu_k)\to f(\uinfty)$ in $L^1_{\text{loc}}(\R^N)$. The two facts combined prove that
	\begin{equation*}
		\int_{B_\bR}\left|f(\tu_k)-f(\uinfty)\right||\uinfty|\to0\qquad\mbox{as}\ \,k\to+\infty
	\end{equation*}
	for all $\bR>0$. Let us fix $\varepsilon>0$. Since $f(\uinfty)\uinfty\in L^1(\R^N)$, there exists $R_0=R_0(\varepsilon)$ such that if $\bR>R_0$ then
	\begin{equation*}
		\left|\int_{\R^N\setminus B_\bR}f(\uinfty)\uinfty\right|<\varepsilon\,.
	\end{equation*}
	Moreover,
	\begin{equation*}
		\begin{split}
			\left|\int_{\R^N\setminus B_\bR}f(\tu_k)\uinfty\right|&\les\int_{\R^N\setminus B_\bR}|\tu_k||\uinfty|+\int_{\R^N\setminus B_\bR}|\tu_k|^p|\uinfty|\\
			&\leq\|\tu_k\|_2\|\uinfty\|_{L^2(\R^N\setminus B_\bR)}+\|\tu_k\|_{pq}^p\|\uinfty\|_{L^{q'}(\R^N\setminus B_\bR)}.
		\end{split}
	\end{equation*}
	Choosing again $q=\frac{p+1}p$, we get $\|\tu_k\|_{pq}\leq C\|\tu_k\|_{H^2(\R^N)}\leq C$ by Proposition \ref{exists_uinfty}. Since $q'=p+1\in(2,2_*)$, there exists $R_1=R_1(\varepsilon)$ such that
	$$\|\uinfty\|_{L^2(\R^N\setminus B_\bR)}<\varepsilon\quad\ \mbox{and}\quad\ \|\uinfty\|_{L^{q'}(\R^N\setminus B_\bR)}<\varepsilon\,.$$
	Hence, if $\bR>\max\{R_0,R_1\}$, one infers that $T_2\leq C\varepsilon$ in \eqref{T_2}, with a constant $C$ independent of $\varepsilon$ if $k$ is large enough. This concludes the proof for the case $N\geq 5$ by letting first $k\to+\infty$ and then $\varepsilon\to0$. The proof in the case $N=4$ is analogous, provided \eqref{f_above_N>=5} is replaced by \eqref{f_above_N=4}. For instance, letting $\alpha,q>0$, we estimate $T_1$ as
	\begin{equation*}
		\int f(\tu_k)|\tu_k-\uinfty|\leq\varepsilon\|\tu_k\|_2\|\tu_k-\uinfty\|_2+C_\varepsilon\int|\tu_k|^q\left(\e^{\alpha\tu_k^2}-1\right)|\tu_k-\uinfty|\,,
	\end{equation*}
	where the first term is estimated as for the case $N\geq5$, and in the second we apply Lemma \ref{estimate_Sani} to get
	\begin{equation*}
		\int|\tu_k|^q\left(\e^{\alpha\tu_k^2}-1\right)|\tu_k-\uinfty|\leq\left(\int|\tu_k|^{qt}\left(\e^{\alpha t\tu_k^2}-1\right)\right)^\frac1t\left(\int|\tu_k-\uinfty|^{t'}\right)^\frac1{t'}
	\end{equation*}
	and the first term is bounded uniformly with respect to $k$ as in the proof of Proposition \ref{exists_uinfty}, while the second converges to $0$ since $\tu_k\to\uinfty$ in $L^{t'}(\R^N)$ for all $t'>2$.
\end{proof}

\begin{proof}[Proof of Theorem \ref{Thm2}]
	By Proposition \ref{exists_uinfty} we know that $\uinfty$ is a weak radial solution of \eqref{eq}, hence it is regular. Since $J_{R_k}'(u_k)u_k=0$ and also $J'(\uinfty)\uinfty=0$, by Lemma \ref{lem_conv} one gets $\iii{u_k}_{R_k}\to\iii{\uinfty}$. This implies first that $\tu_k\to\uinfty$ in $H^2(\R^N)$, and moreover that
	\begin{equation}\label{conv_MPlevel}
		c_\infty:=\lim_{k\to+\infty}c_{R_k}=\lim_{k\to+\infty}J_{R_k}(u_k)=J(\uinfty).
	\end{equation}
	Since $c_R>\delta>0$ by Lemma \ref{MPlevel_bdd}, then $J(\uinfty)=c_\infty\geq\delta$, and therefore $\uinfty$ is nontrivial.
	
	Let us prove first that $\uinfty\geq0$ in $\R^N$. Note that so far we just used the fact that $\iii{\cdot}$ is a norm on $H^2(\R^N)$, i.e. the bound $\gamma>-2\sqrt\gamma_0$. Now, in order to use Lemma \ref{Pos_CT}, we need to restrict to $\gamma\geq0$. Let $\mu_B:=\mu_0(B_1)$ and scale the problem in $B_R$ to the corresponding in $B_1$. In other words, if $u_R\in H^2_{0,\rad}(B_R)$ is a solution of \eqref{eq_R}, define $v\in H^2_{0,\rad}(B_1)$ as $v(\cdot):=u(R\cdot)$. Then $v$ solves
	\begin{equation}\label{eq_1_v}
		\begin{cases}
			\Delta^2v-R^2\gamma\Delta v+R^4\gamma_0 v=R^4f(v)&\;\mbox{in}\;B_1,\\
			v=\partial_\nu v=0&\;\mbox{on}\;\dB_1.
		\end{cases}
	\end{equation}
	In order for \eqref{eq_1_v} to enjoy positivity by Lemma \ref{Pos_CT}, one should require $R^2\gamma+R^4\gamma_0>\mu_B$, which is always true for large enough $R$ for $\gamma\geq0$ and $\gamma_0>0$ fixed. This implies that there exists $R_0=R_0(\gamma,\gamma_0,N)>0$ such that the mountain-pass solution $u_R$ of \eqref{eq_R} at level $c_R$ is positive in $B_R$ for $R\geq R_0$. By pointwise convergence, it is then immediate to infer that $u_\infty\geq0$ in $\R^N$. 
	
	\vskip0.2truecm
	 In order to conclude the proof, we show that $\uinfty$ is a \textit{radial ground state} solution for \eqref{eq}.
	To this aim, we are going to construct a mountain-pass path $\gamma_R \in C([0,1], H^2_{0,\rad}(B_R))$ such that
	\begin{enumerate}[($f_1$)]
		\item[(i)] $J_R(\gamma_R(0))=0$, $J_R(\gamma_R(1))<0$;
		\item[(ii)] for some $t_R\in (0,1)$ one has
		\begin{equation}\label{4.7bis}
			c_R\leq J_R(\gamma_R(t_R))=\max_{t\in[0,1]}J_R(\gamma_R(t))=:C_R
		\end{equation}	
		and $\lim\limits_{R\rightarrow+\infty}C_R= c_{min}$, where
		$$c_{min}:=\inf_{u\in\cN}J(u)$$
		is the least energy in $H^2_{\rad}(\R^N)$.
	\end{enumerate}
	Note that under assumptions ($f_1$)-($f_5$) the Nehari manifold associated to the functional $J$
	\begin{equation}\label{4.7ter}
		\cN:=\left\{u\in H_{\rad}^2(\R^N)\setminus\{0\}\,|\,J'(u)u=0\right\}
	\end{equation}
	is well-defined and contains all critical points of $J$. Moreover, it is easy to infer that for any $u\in H^2(\R^N)\setminus\{0\}$ there exists a unique $t_u>0$ such that
	\begin{equation}\label{max_subcritical}
		J(t_uu)=\max_{t>0}J(tu)\quad\text{and}\quad t_uu\in\cN.
	\end{equation}
	The existence of a radial ground state solution for $J$ can be obtained by following the same arguments as in \cite[Theorem 4.23]{Rabinowitz92}, as in \cite[Lemma 2.3]{Pimenta14} for the case of polynomial nonlinearities in dimension $N\geq5$, and\footnote{The fact that the radial mountain-pass solution found in \cite{Sani_critical} is a radial ground state follows from ($f_5$) as in \cite[Section 5]{dOM}.} \cite[Theorem 1.1]{Sani_subcritical} for the conformal case $N=4$. Let $\bar{u}\in H^2(\R^N)$ be such a radial ground state solution of \eqref{eq}. We define $\bar{u}_R:=\phi_R\bar{u}$, where $\phi_R:=\phi\left(\frac\cdot R\right)$ with $\phi\in C_0^\infty(\R^N,[0,1])$ a radial cut-off function such that $\phi(|x|)=1$ for $|x|\leq
	\frac12$, $\phi(|x|)\in(0,1)$ for $\frac12<|x|<1$, and $\phi(|x|)=0$ for $|x|\geq 1$. Hence, $\bar{u}_R\in H^2_{0,\rad}(B_R)$ and $\bar u_R\to\bar u$ in $H^2(\R^N)$.
	Choose $T_0>0$ large enough such that $J(T_0\bar{u})<0$, then we define the path $\gamma_R$ as follows
	$$\gamma_R(t):=tT_0\bar{u}_R,\quad \text{for}\,\, t\in[0,1]\,\,\text{and}\,\, R>0\,\,\text{large\,enough}.$$
	It is obvious that $J_R(\gamma_R(0))=0$, $J_R(\gamma_R(1))<0$, that is, conclusion (i) holds true. Moreover, $\gamma_R(t)$ is continuous in $t$. As a result, there exists $t_R\in(0,1)$ such that \eqref{4.7bis} holds. Recalling \eqref{max_subcritical} and the fact that $\bar{u}\in H^2(\R^N)$ is a radial ground state solution of \eqref{eq}, we can obtain immediately
	\begin{equation}\label{max1}
		c_{min}=J(\bar{u})=\max_{t>0}J(t\bar{u}).
	\end{equation}
	Since $\bar u_R\to\bar u$ in $H^2(\R^N)$, it then follows that
	$$\lim_{R\to+\infty}C_R=\lim_{R\to+\infty}\max_{t\in[0,1]}J_R(tT_0\bar{u}_R)=\max_{t\in[0,1]} J(tT_0\bar{u})\leq c_{min},$$
	which implies by \eqref{conv_MPlevel} and \eqref{4.7bis} that
	\begin{equation}\label{cmin}
		c_\infty=\lim_{R\to+\infty}c_R\leq\lim\limits_{R\to+\infty}C_R\leq c_{min}.
	\end{equation}
	Recall that $u_\infty$ is a weak solution of \eqref{eq} with $J(u_\infty)=c_\infty$. Then by the definition of $c_{min}$, we have $c_\infty\geq c_{min}$, which, together with \eqref{cmin}, yields $c_\infty=c_{min}$. Lemma \ref{lem:no-zero-spheres} now implies $\uinfty(x)>0$ for every $x\ne0$. If $\gamma^2\geq4\gamma_0$, Proposition \ref{prop:factorization-positive} also gives $\uinfty(0)>0$. This concludes the proof.
\end{proof}

\section{The critical case for $\boldsymbol{N\geq5}\,$: Proof of Theorem \ref{Thm-C-5-general}}\label{Section_crit_largedim}
The expanding-domain strategy of Section \ref{Section_Subcritical} also
applies to critical growth after the required compactness estimates are
established. The energy comparison needs the existence of a radial ground
state of \eqref{eq}; this permits a special mountain-pass path and identifies
the limiting level with the least radial energy. Lemma
\ref{lem:no-zero-spheres} can then be applied to obtain positivity away from
the origin. We first establish the needed ground-state existence in the
setting of Theorem \ref{Thm-C-5-general}.

In this case, the corresponding energy functional is
$$
J(u)=\frac12\iii{u}^2-\int \bigg(G(u)+\frac{u_+^{2_*}}{2_*}\bigg),
$$
where we recall that $2_*:=\frac{2N}{N-4}$, $G(u):=\int^u_0g(t)\dd t$, $u\in H_{\rad}^2(\R^N)$ and its Nehari manifold $\cN$ is defined in \eqref{4.7ter}. It is clear that all critical points of $J$ must lie on $\cN$. Moreover, it is not difficult to verify that the energy functional $J$ is bounded below and coercive on $\cN$. Indeed, for any fixed $u\in\cN$,
\begin{equation}\label{bound-below}
	\aligned
		J(u)&=J(u)-\frac{1}{\mu}J'(u)u=\frac{\mu-2}{2\mu}\iii{u}^2
		+\frac{2_*-\mu}{2_*\mu}\int u_+^{2_*}+\int \bigg(\frac{1}{\mu}g(u)u-G(u)\bigg),
	\endaligned
\end{equation}
which implies the desired conclusion by ($f_4$). Note that ($f_1$) and ($f_5$) imply
$$
g'(t)t^2-g(t)t\geq0\quad\text{for}\,\,t\geq0.
$$
Define now $I:H_{\rad}^2(\R^N)\to\R$ as
$$I(u):=J'(u)u=\iii{u}^2-\int(g(u)u+u_+^{2_*}),$$
then for any $u\in \cN$,
\begin{equation}\label{condmim}
	\aligned
		I'(u)u&=2\iii{u}^2-\int \bigg((g'(u)u^2+g(u)u)+2_*u_+^{2_*}\bigg)\\
		&=-\int \bigg[(g'(u)u^2-g(u)u)+(2_*-2)u_+^{2_*}\bigg]<0\,.
	\endaligned
\end{equation}
This implies that $\cN$ is a $C^1$ manifold with codimension $1$, and thus that the Nehari manifold is a natural constraint for the functional $J$. 
 Since $J$ is bounded from below on $\cN$ by \eqref{bound-below}, we may define its minimal energy as
$$
c_{min}:=\inf_{u\in\cN}J(u).
$$
Observe that for $u\in\cN$, by the Sobolev embedding \cite{Lions} and ($f_3'$) one has for any $\varepsilon>0$, there exists $C_\varepsilon>0$ such that
$$
\iii{u}^2=\int(g(u)u+u_+^{2_*})\les \varepsilon\iii{u}^2 +C_\varepsilon\iii{u}^q+\iii{u_+}^{2_*},
$$
which implies that there exists $C>0$ independent of $u\in \cN$ such that $\iii{u}^2\geq C$. This, combined with \eqref{bound-below}, yields $c_{min}>\frac{\mu-2}{2\mu} C>0$. Moreover, it is easy to obtain that for any $u\in H^2(\R^N)\setminus\{0\}$, there exists a unique $t_u>0$ such that
\begin{equation}\label{max}
	J(t_uu)=\max_{t>0}J(tu)\quad\text{and}\quad t_uu\in\cN.
\end{equation}

As usual, we define the homogeneous Sobolev space $D^{2,2}_0(\R^N)$ as the completion of $C^\infty_0(\R^N)$ with respect to the norm\footnote{Actually, the space $D^{2,2}_0(\R^N)$ is more commonly defined as the completion with respect to the norm $\|D^2\cdot\|_2$, i.e. with respect to the full tensor of all second derivatives. However, the two definitions are equivalent since, by integration by parts, $\|D^2\cdot\|_2$ and $\|\nabla^2\cdot\|_2$ are equivalent norms on $C^\infty_0(\R^N)$, see e.g. \cite[Sec.2.2.1]{GGS}.}
\begin{equation*}
	\|u\|_{D^{2,2}_0(\R^N)}:=\left(\int_{\R^N}|\Delta u|^2\right)^\frac12.
\end{equation*}
It is well-known that $D^{2,2}_0(\R^N)\hookrightarrow L^{2_*}(\R^N)$ holds for $N\geq5$, and we denote by $S$ its best constant, namely
\begin{equation}\label{S}
	S:=\inf\left\{\|\Delta u\|_2^2\,\,\big|\, u\in D^{2,2}_0(\R^N),\,\|u\|_{2_*}=1\right\}.
\end{equation}
\begin{prop}\label{lem_conv_crit}
	If $c_{min}<\frac2NS^\frac N4$, then there exists $\bu\in\cN$ such that $J(\bu)=c_{min}$. Moreover, $\bu$ is a radial ground state solution of \eqref{eq}.
\end{prop}
\begin{proof}
	Take a sequence $(u_n)_n\subset\cN$ such that $J(u_n)\to c_{min}$ as $n\to+\infty$. It is easy to verify that $(u_n)_n$ is bounded in $H^2(\R^N)$ uniformly for $n$. Therefore, there exists $\bu\in H_{\rad}^2(\R^N)$ such that, up to a subsequence,
	$$
	\aligned
		&u_n\rightharpoonup\bu\quad\text{in}\,\,H^2(\R^N),\\
		&u_n\to\bu\quad\text{in}\,\,L^s(\R^N),\,\,s\in(2,2_*),\\
		&u_n\to\bu\quad\mbox{a.e.}\ \text{in}\ \,\R^N.
	\endaligned
	$$
	If $\bu\equiv0$, then from the above convergences and $J'(u_n)u_n=0$, one gets
	$\iii {u_n}^2=\int u_n^{2_*}+o_n(1)$. So, recalling
	$$
	{\frac{\|\Delta u_n\|^2}{\big(\int u_n^{2_*}\big)^{\frac2{2_*}}}\geq S}
	$$
	we deduce $\iii {u_n}^2\geq S^\frac N4+o_n(1)$. It then follows from $(u_n)_n\subset\cN$ that
	$$
	\aligned
		\frac2NS^\frac N4&>c_{min}=J(u_n)+o_n(1)\\
		&=J(u_n)-\frac{1}{2_*}J'(u_n)u_n+o_n(1)\\
		&=\left(\frac12-\frac1{2_*}\right)\iii {u_n}^2+o_n(1)\geq\frac2{N}S^\frac N4.
	\endaligned
	$$
	This is a contradiction, hence necessarily $\bu\not\equiv0$. Using Fatou's lemma we obtain
	\begin{equation}\label{fatou}
		\aligned
			c_{min}+o_n(1)&=J(u_n)-\frac{1}{\mu}J'(u_n)u_n\\
			&=\frac{\mu-2}{2\mu}\iii{u_n}^2+\frac{2_*-\mu}{2_*\mu}\int (u_n)_+^{2_*}+\int \bigg(\frac{1}{\mu}g(u_n)u_n-G(u_n)\bigg)\\
			&\geq\frac{\mu-2}{2\mu}\iii{\bu}^2+\frac{2_*-\mu}{2_*\mu}\int \bu_+^{2_*}+\int\bigg(\frac1\mu g(\bu)\bu-G(\bu)\bigg)+o_n(1).
		\endaligned
	\end{equation}
	Moreover, by \eqref{max} there exists a unique $t_\bu>0$ such that $t_\bu\bu\in \cN$. Recalling the definition of $c_{min}$, we get
	$$
	\aligned
		J(t_\bu\bu)&=J(t_\bu\bu)-\frac1\mu J'(t_\bu\bu)t_\bu\bu\\
		&=\frac{\mu-2}{2\mu}t_\bu^2\iii{\bu}^2+\frac{2_*-\mu}{2_*\mu}t_\bu^{2_*}\int \bu_+^{2_*}+
		\int\bigg(\frac1\mu g(t_\bu\bu)t_\bu\bu-G(t_\bu\bu)\bigg)\\
		&\geq c_{min},
	\endaligned
	$$
	which implies by \eqref{fatou} that $t_\bu\geq1$.
	
	Set now $v_n:=u_n-\bu$, then it follows by Brezis-Lieb lemma, see \cite{BrLi}, that
	\begin{equation}\label{BL}
		\aligned
			&\iii {u_n}^2=\iii {v_n}^2+\iii\bu^2+o_n(1),\\
			&\int |{u_n}|^2=\int |{v_n}|^2+\int|\bu|^2+o_n(1),\\
			&\int G(u_n)=\int G(v_n)+\int G(\bu)+o_n(1),\\
		    &\int g(u_n)u_n=\int g(v_n)v_n+\int g(\bu)\bu+o_n(1),\\
			&\int (u_n)_+^{2_*}=\int (v_n)_+^{2_*}+\int\bu_+^{2_*}+o_n(1).
		\endaligned
	\end{equation}

	If $t_\bu=1$, then $\bu\in\cN$ and therefore it is a ground state by using Fatou's Lemma. If, instead, $t_\bu>1$, since $\bu\not\equiv0$, we claim that $J(\bu)>0$ and $J'(\bu)\bu>0$. { Indeed,
	\begin{equation*}
		J(tu)=t^2\left(\frac12\iii u^2-\int\frac{G(tu)}{t^2}-\frac{t^{2_*-2}}{2_*}\int|u|^{2_*}\right)
	\end{equation*}
	and note that the expression in the brackets is decreasing in $t$ by ($f_4$). Since $t_\bu>1$, then we obtain $J(\bu)>J(t_{\bu}\bu)>0$. On the other hand,
	\begin{equation*}
		J'(tu)tu=t^2\left(\iii u^2-\int\frac{g(tu)tu}{t^2}-t^{2_*-2}\int|u|^{2_*}\right)
	\end{equation*}
	and the second term in the product is decreasing in $t$ by ($f_5$), so $J'(\bu)\bu>J'(t_{\bu}\bu)t_{\bu}\bu=0$.}
	Hence, recalling $(u_n)_n\subset\cN$, from \eqref{BL} we infer
	\begin{equation}\label{splitting_Jvn}
		\begin{split}
			J'(v_{n})v_n&=J'(u_n)u_n-J'(\bu)\bu+o_n(1)<0.
		\end{split}
	\end{equation}
	Moreover, since $v_n\to0$ in $L^s(\R^N)$ for all $s\in(2,2_*)$, recalling ($f_2$) and ($f'_3$), see \eqref{f_above_N>=5}, one also has
	\begin{equation}\label{vanish}
		\int G(t_\bu v_n)=o_n(1)\quad\ \mbox{and}\ \quad\int g(t_\bu v_n)t_\bu v_n=o_n(1).
	\end{equation}
	Hence \eqref{splitting_Jvn} and \eqref{vanish} imply
	\begin{equation}\label{tu0}
		\iii{v_n}^2\leq\int|(v_n)_+|^{2_*}+o_n(1)
	\end{equation}
	Note that if $\iii {v_n}\geq C>0$ for sufficiently large $n$, then
	$$
	\frac{\iii {v_n}^2}{\big(\int |(v_n)_+|^{2_*}\big)^{\frac2{2_*}}}\geq S
	$$
	and this, together with \eqref{tu0}, yields $\iii {v_n}^2\geq S^\frac N4+o_n(1)$. Then by \eqref{BL}-\eqref{tu0} and\footnote{Note that we can always assume that $\mu<2_*$ in ($f_4$).} ($f_4$) one obtains
	$$
	\aligned
		\frac2NS^\frac N4&> J(u_n)+o_n(1)-\frac{1}{\mu}J'(u_n)u_n\\
		&=\frac{\mu-2}{2\mu}\iii{u_n}^2+\frac{2_*-\mu}{2_*\mu}\int (u_n)_+^{2_*}
		+\int \bigg(\frac{1}{\mu}g(u_n)u_n-G(u_n)\bigg)+o_n(1)\\
		&=\frac{\mu-2}{2\mu}\iii{\bu}^2+\frac{2_*-\mu}{2_*\mu}\int \bu_+^{2_*}+\int\bigg(\frac1\mu g(\bu)\bu-G(\bu)\bigg)+\frac{\mu-2}{2\mu}\iii{v_n}^2\\
		&\quad+\frac{2_*-\mu}{2_*\mu}\int(v_n)_+^{2_*}+o_n(1)\\
		&>\frac{\mu-2}{2\mu}\iii{v_n}^2+\frac{2_*-\mu}{2_*\mu}\int(v_n)_+^{2_*}+o_n(1)\\
		&\geq \bigg(\frac12-\frac{1}{2_*}\bigg)\iii{ v_{n}}^2\geq\frac2N S^\frac N4+o_n(1),
	\endaligned
	$$
	which is a contradiction. Thus, $\iii{v_n}=o_n(1)$, that is $u_n\to\bu$ in $H^2(\R^N)$, which in turns implies $J(\bu)=c_{min}$ and $\bu\in\cN$. It remains to prove that $\bu$ is a ground state solution of \eqref{eq}. Since $\cN$ is a $C^1$ manifold with codimension $1$ and $J(\bu)=c_{min}$ and $\bu\in\cN$, there exists a Lagrange multiplier $\lambda\in\R$ such that $J'(\bu)\bu=\lambda I'(\bu)\bu$. Since the left-hand side vanishes for all elements in $\cN$, and $I'(u)u<0$ for all $u\in\cN$ by \eqref{condmim}, we conclude that $\lambda=0$, which implies $J'(\bu)=0$ in $H^2(\R^N)$. The proof is complete.
\end{proof}

Observe from Proposition \ref{lem_conv_crit} that in order to prove the existence of a ground state solution $\bu\in H^2(\R^N)$ of \eqref{eq}, we still need to show that the minimal energy value satisfies $c_{min}<\frac2NS^\frac N4$. Here the assumptions on $p$ and $\nu$ in ($f_3'$) will play an essential role.

\begin{lem}\label{comp}
	Assume all conditions of Theorem \ref{Thm-C-5-general} hold, then $c_{min}<\frac2NS^\frac N4$.
\end{lem}
\begin{proof}
	It is well-known that the best constant $S$ defined in \eqref{S} is attained only by the functions $k U_{\varepsilon,x_0}$ for $k\in\R\setminus\{0\}$ and $U_{\varepsilon,x_0}$ defined by
	$$U_{\varepsilon,x_0}(x)=\frac{[N(N-4)(N^2-4)\varepsilon^2]^{\frac{N-4}8}}{(\varepsilon+|x-x_0|^2)^{\frac{N-4}2}},\quad\forall x_0\in\R^N,\,\forall\varepsilon>0\,.$$
	Indeed,
	$$\|\Delta U_{\varepsilon,x_0}\|_2^2=\|U_{\varepsilon,x_0}\|_{2_*}^{2_*}=S^{N/4},$$
	see \cite{Van93,EFJ90}. Let $\phi\in C_0^\infty(\R^N,[0,1])$ be a radial cut-off function such that $\phi(x)=1$ for $|x|\leq1$, $\phi(x)\in(0,1)$ for $1<|x|<2$, and $\phi(x)=0$ for $|x|\geq2$. Set
	\begin{equation}\label{function}
		\psi_\varepsilon(x)=\phi(x)U_{\varepsilon,0}(x).
	\end{equation}
	$\bullet$ By \cite{GDW94,He21}, when $N\geq 8$, as $\varepsilon\to0^+$, the following expansions hold:
	\begin{equation}\label{compu1}
		\aligned
			&\|\Delta \psi_\varepsilon\|_2^2=S^\frac N4+O(\varepsilon^{\frac{N-4}2})\,,\\
			&\|\nabla \psi_\varepsilon\|_2^2=C\varepsilon+O(\varepsilon^{\frac{N-4}2})\,,\\
			&\|\psi_\varepsilon\|_{2_*}^{2_*}=S^\frac N4+O(\varepsilon^\frac N2)\,,\\
		\endaligned
	\end{equation}
	and
	\begin{equation}\label{compu2}
	\|\psi_\varepsilon\|_2^2=\left\{
	  \begin{array}{ll}
	          C\varepsilon^2+ O(\varepsilon^{\frac{N-4}2})& \mbox{for}\,\, N>8\,,\\
	          -C\varepsilon^2\ln\varepsilon+O(\varepsilon^2), & \mbox{on}\,\, N=8\,.
	  \end{array}
	\right.
	\end{equation}
	Moreover,
	\begin{equation}\label{compu4}
		\aligned
			\|\psi_\varepsilon\|_p^p&\geq C\int_{|x|\leq 1}\frac{\varepsilon^{\frac{N-4}4p}}{(\varepsilon+|x|^2)^{\frac{(N-4)p}2}}\dd x= C\varepsilon^{\frac{N-4}4p-\frac{N-4}2p}\int_{|x|\leq1}\frac{\dd x}{(1+|\frac{x}{\sqrt{\varepsilon}}|^2)^{\frac{(N-4)p}2}}\\
			&=C\varepsilon^{\frac{4-N}4p+\frac N2}\int_0^1\frac{t^{N-1}}{(1+t^2)^{\frac{(N-4)p}2}}\dd t=C\varepsilon^{\frac{4-N}4p+\frac N2}.
		\endaligned
	\end{equation}
	Now we claim that $c_{min}<\frac2NS^\frac N4$ for any fixed $\eta>0$
	and $p\in\left(\frac{2(N-2)}{N-4},\frac{2N}{N-4}\right)$, or provided $\eta\in(\eta_0,+\infty)$ with some $\eta_0>0$ and $p\in\left(2,\frac{2(N-2)}{N-4}\right]$. Since $\psi_\varepsilon\in H^2(\R^N)\setminus\{0\}$, from \ref{max} that
	there exists $t_\varepsilon>0$ such that
	$$
	J(t_\varepsilon \psi_\varepsilon)=\max_{t>0}J(t\psi_\varepsilon)\quad\text{and}\quad t_\varepsilon\psi_\varepsilon\in\cN.
	$$
	Recalling \eqref{compu1} and \eqref{compu2}, it is easy to see that
	$\psi_\varepsilon$ is bounded in $H^2(\R^N)$ uniformly for $\varepsilon\in(0,1]$.
	Therefore, since $p\in(2,\frac{2N}{N-4})$ there exist constants $c_2>c_1>0$ independent of $\varepsilon$ such that
	\begin{equation}\label{compu6}
	J(t\psi_\varepsilon)< \frac2{N}S^\frac N4\quad\text{for}\,\, t<c_1\,\,\text{or}\,\, t>c_2.
	\end{equation}
	For $t\in[c_1,c_2]$ and $N\geq 8$, for any fixed $\eta>0$ one gets
	\begin{equation}\label{compu7_0}
		\aligned
			J(t\psi_\varepsilon)&\leq\max_{t>0}\left\{\left(\frac{t^2}2\int|\Delta \psi_\varepsilon|^2-\frac{t^{2_*}}{2_*}\int |\psi_\varepsilon|^{2_*}\right)+\frac{t^2}2\int(\gamma|\nabla \psi_\varepsilon|^2+\gamma_0\psi_\varepsilon^2)-\frac{t^p\eta}{p}\int|\psi_\varepsilon|^p\right\}\\
			&\leq\frac2N\frac{\|\Delta \psi_\varepsilon\|_2^\frac N2}
			{\|\psi_\varepsilon\|_{2_*}^\frac N2}+\frac{c_2^2}2\int(\gamma|\nabla\psi_\varepsilon|^2+\gamma_0\psi_\varepsilon^2)-\frac{c_1^p\eta}{p}\int|\psi_\varepsilon|^p
		\endaligned
	\end{equation}
	Then, by \eqref{compu1}, \eqref{compu2}, and \eqref{compu4}, for a sufficiently small $\varepsilon>0$,
	\begin{equation}\label{compu7}
		\begin{split}
			J(t\psi_\varepsilon)&\leq\frac2NS^\frac N4\frac{[1+O(\varepsilon^{\frac{N-4}2})]^{\frac N4}}{[1+O(\varepsilon^\frac N2)]^{\frac{(N-4)}4}}+\bigg(\frac{c_2^2\gamma C}2\varepsilon+\frac{c_2^2\gamma_0 C}2\varepsilon^{3/2}\bigg)-C\eta\, \varepsilon^{\frac{4-N}4p+\frac N2}\\
			&\leq\frac2NS^\frac N4[1+O(\varepsilon^{\frac{N(N-4)}8})]+C\varepsilon-C\eta\,\varepsilon^{\frac{4-N}4p+\frac N2}\\
			&<\frac2NS^\frac N4,
		\end{split}
	\end{equation}		
	which holds for all $\eta>0$ if $p\in\left(\frac{2(N-2)}{N-4},\frac{2N}{N-4}\right)$ or, for $\varepsilon$ small and fixed, if $\eta>\eta_0(\varepsilon)$ if $p\in\left(2,\frac{2(N-2)}{N-4}\right]$. In both cases, we conclude that
	\begin{equation}\label{c_min}
		c_{min}\leq\max_{t>0}J(t\psi_\varepsilon)<\frac2NS^\frac N4
	\end{equation}
	as desired.
	{Let us now focus on the remaining dimensions $N=5,6,7$, where we consider
	\begin{equation}\label{compu8}
	\phi(x)=\phi(|x|)\in C^\infty(\overline{B_1(0)},[0,1]), \,\,\phi(0)=1,\,\,\phi(1)=\phi'(1)=0.
	\end{equation}
	When $N=7$, in addition to \eqref{compu8}, the function $\phi$ is chosen so that
	$$
	\frac{|\phi'(t)|^2}{t^2}\leq C\,,\quad |\phi^2(t)-1|\leq Ct^{3+\delta}, \quad
	\big||\phi(t)|^{\frac{14}3}-1\big|\leq Ct^{3+\delta}
	$$
	for some fixed $\delta\in(0,1)$. Then the function $\psi_\varepsilon$ given in \eqref{function} satisfies
	\begin{equation}\label{compu9}
		\aligned
		&\|\Delta \psi_\varepsilon\|_2^2=S^\frac74+C\varepsilon^\frac32+O(\varepsilon^{\frac{3+\delta}2}),\\
		&\|\nabla\psi_\varepsilon\|_2^2=C\varepsilon^{\frac32}+O(\varepsilon^2),\\
		&\|\psi_\varepsilon\|_{2_*}^{2_*}=S^\frac74+O(\varepsilon^{\frac{3+\delta}2}),\\
		&\|\psi_\varepsilon\|_2^2=C \varepsilon^{\frac32}+O(\varepsilon^2).
		\endaligned
	\end{equation}
	as $\varepsilon\to0^+$, see \cite[Lemma 3.7]{He21}. Note that, in order to obtain the estimate from below \eqref{compu4}, it is sufficient that $\phi(x)>\delta>0$ in $B_{r_0}(0)$ for some $\delta,r_0>0$, which follows from continuity and $\phi(0)>0$. Indeed,
	\begin{equation}\label{compu4-1}
		\|\psi_\varepsilon\|_p^p\geq\delta^p\int_{B_{r_0}(0)}|U_{\varepsilon,0}(x)|^p\dd x=C\delta^p\int_{|x|\leq r_0}\frac{\varepsilon^{\frac{N-4}4p}}{(\varepsilon+|x|^2)^{\frac{(N-4)p}2}}\dd x=C\varepsilon^{\frac{4-N}4p+\frac N2}.
	\end{equation}
	as in \eqref{compu4}. Arguing as before, from \eqref{compu7_0}, \eqref{compu9}, and \eqref{compu4-1} we obtain
	\begin{equation*}
		J(t\psi_\varepsilon)\leq\frac27\frac{\left(S^\frac74(1+O(\varepsilon^{3/2}))\right)^{7/4}}{\left(S^\frac74(1+O(\varepsilon^\frac{3+\delta}2))\right)^{3/4}}+C\varepsilon^\frac32-C\eta\,\varepsilon^{\frac72-\frac34p},
	\end{equation*}
	from which again \eqref{c_min} holds, provided $p>\frac83$, or $p\leq\frac83$ and $\eta$ large as before.
	
	We proceed similarly for dimensions $N=5$ and $N=6$. In the former case, in addition to \eqref{compu8} the function $\phi$ must fulfill
	$$
	\frac{|\phi'(t)|^2}t\leq C\,,\quad |\phi^2(t)-1|\leq Ct^{2+\delta}, \quad\ \big|\phi^6(t)-1\big|\leq Ct^{2+\delta}
	$$
	for some fixed $\delta\in(0,1)$, so that we obtain the following expansions
	\begin{equation*}
		\aligned
		&\|\Delta \psi_\varepsilon\|_2^2=S^\frac32+C\varepsilon+O(\varepsilon^{\frac{2+\delta}2})\,,\\
		&\|\nabla\psi_\varepsilon\|_2^2=C\varepsilon+O(\varepsilon^\frac32),\\
		&\|\psi_\varepsilon\|_{2_*}^{2_*}=S^\frac32+O(\varepsilon^{\frac{2+\delta}2})\,,\\
		&\|\psi_\varepsilon\|_2^2=C\varepsilon+O(\varepsilon^\frac32)\,,
		\endaligned
	\end{equation*}
	as $\varepsilon\to0^+$, see \cite[Lemma 3.5]{He21}. In the latter case $N=5$, $\phi$ is chosen so that \eqref{compu8} and
	$$
	|\phi^2(t)-1|\leq Ct^{1+\delta}\,,\quad\big||\phi(t)|^{10}-1\big|\leq Ct^{1+\delta}
	$$
	for some fixed $\delta\in(0,1)$ hold, so that one gets
	\begin{equation*}
		\aligned
		&\|\Delta \psi_\varepsilon\|_2^2=S^\frac54+C\varepsilon^\frac12+O(\varepsilon^{\frac{1+\delta}2})\,,\\
		&\|\nabla \psi_\varepsilon\|_2^2=C\varepsilon^\frac12+O(\varepsilon)\,,\\
		&\|\psi_\varepsilon\|_{2_*}^{2_*}=S^\frac54+O(\varepsilon^{\frac{1+\delta}2})\,,\\
		&\|\psi_\varepsilon\|_2^2=C\varepsilon^\frac12+O(\varepsilon).
	\endaligned
	\end{equation*}
	as $\varepsilon\to0^+$, see \cite[Lemma 3.3]{He21}. In both cases, one concludes as before, under the conditions on $p$ and $\eta$ given in Theorem \ref{Thm-C-5-general}.
	
	Note that for all cases $N=5,6,7$, examples of functions $\phi$ with the above properties have been provided in Lemmas 3.4, 3.6, 3.8, of \cite{He21} respectively.}
\end{proof}

\begin{proof}[Proof of Theorem \ref{Thm-C-5-general}]
	In view of Proposition \ref{lem_conv_crit} and Lemma \ref{comp}, there exists a radial ground state solution $\bu\in H^2(\R^N)$ of equation \eqref{eq} with $J(\bu)=c_{min}$. As for the subcritical case, let us now first fix $R>0$ and consider the biharmonic Dirichlet problem in the ball $B_R$
	\begin{equation}\label{eq_R+C}
		\begin{cases}
			\Delta^2u-\gamma\Delta u+\gamma_0 u=g(u)+u_+^{2_*}&\;\mbox{in}\;B_R,\\
			u=\partial_\nu u=0&\;\mbox{on}\;\dB_R,
		\end{cases}
	\end{equation}
	with $\gamma\geq0$ and $\gamma_0>0$. The corresponding energy functional is
	$$
	J_R(u)=\frac12\iii{u}_R^2-\int_{B_R}\left(G(u)+\frac{u_+^{2_*}}{2_*}\right)
	$$
	restricted to $H^2_{0,\rad}(B_R)$. Since $p\in(2, 2_*)$, it is easy to see that $J_R(tu)\to-\infty$ for all $u\in H^2_0(B_R)$ fixed as $t\to+\infty$, and by Sobolev inequality, one has $J_R(u)>\delta_R>0$ on $S_\rho^R:=\{u\in H^2_{0,\rad}(B_R)\,|\,\iii{u}_R=\rho_R\}$ for some $\delta_R,\rho_R>0$. Hence we can define the mountain-pass level
	$$c_R:=\inf_{\gamma\in\Gamma_R}\max_{u\in\gamma([0,1])}J_R(u),$$
	where
	$$\Gamma_R:=\left\{\gamma\in C\left([0,1],H^2_{0,\rad}(B_R)\right)\,|\,\gamma(0)=0,\gamma(1)=v_0^R\right\}$$
	and $v_0^R\in H^2_{0,\rad}(B_R)$ is such that $J_R(v_0^R)<0$. Similarly to the subcritical case, $v_0^R$ and the constants $\delta_R$ and $\rho_R$ can be chosen independently of $R$, see Lemma \ref{MPunif}. Moreover, there exists a Palais-Smale sequence $(u_n)_n\subset H^2_{0,\rad}(B_R)$ such that $J_R(u_n)\rightarrow c_R$ and $J'_R(u_n)\rightarrow 0$ in $\left(H^2_{0,\rad}(B_R)\right)'$ as $n\to+\infty$. One can construct a special mountain-pass path by choosing $\psi_\varepsilon\in H^2_{0,\rad}(B_R)\setminus\{0\}$ defined in Lemma \ref{comp} for $R>0$ large, and set $\gamma(t):=tT\psi_\varepsilon$ for $T$ large enough and $t\in[0,1]$. Then it follows from \eqref{max} and the definition of $c_R$ that there exists $t_\varepsilon>0$ such that
	$$
	c_R\leq\max_{t\in[0,1]}J_R(\gamma(t))={\max_{t>0}J_R(tT\psi_\varepsilon)=J_R(t_\varepsilon T\psi_\varepsilon)}.
	$$
	As in Lemma \ref{comp}, we deduce that $c_R<\frac2NS^\frac N4$ uniformly for large $R>0$, and then as in Proposition \ref{lem_conv_crit} one obtains $u_n\rightarrow u_R$ in $H^2_{0,\rad}(B_R)$ with $J'_R(u_R)=0$ and $J_R(u_R)=c_R$, which implies that $u_R$ is a mountain-pass solution of equation \eqref{eq_R+C}. In order to find a solution of \eqref{eq}, let now $R\to+\infty$. In view of Corollary \ref{MPlevel_bdd}, the mountain-pass level map $R\mapsto c_R$ is decreasing and $0<\delta<c_R\leq c_1$ for all $R\geq1$. Hence, it is not hard to prove that the (extended) family of mountain-pass solutions $(\tu_R)_R$ is uniformly bounded in $H^2_{\rad}(\R^N)$. Furthermore, arguing again as in Proposition \ref{lem_conv_crit}, one infer the existence of $u_\infty\in H^2_{\rad}(\R^N)$ and a subsequence $(\tu_k)_k$ with $\tu_k:=\tu_{R_k}$ such that 
	\begin{equation}\label{conver}
		\tu_k\to u_\infty\ \ \mbox{in}\ \,H^2(\R^N)\quad\text{and}\quad c_{R_k}\to c_\infty\quad\text{as}\ \,k\to+\infty.
	\end{equation}
	with $c_\infty\geq\delta$. Thus, $u_\infty$ is a weak solution of \eqref{eq} with $J(u_\infty)=c_\infty$. Indeed,
	\begin{equation}
		c_\infty=\lim_{k\to+\infty}c_{R_k}=\lim_{k\to+\infty}J_{R_k}(u_k)=\lim_{k\to+\infty}J(\tu_k)=J(\uinfty),
	\end{equation}
	using the continuity of $J$ in $H^2(\R^N)$ and \eqref{conver}. Comparing
	the expanding-ball mountain-pass levels with the radial ground state $\bu$,
	exactly as in the last part of the proof of Theorem \ref{Thm2}, gives
	$J(\uinfty)=c_{min}$. Lemma \ref{lem:no-zero-spheres} therefore yields
	$\uinfty(x)>0$ for every $x\ne0$. If $\gamma^2\geq4\gamma_0$,
	Proposition \ref{prop:factorization-positive} yields positivity also at the
	origin.
\end{proof}

\section{The critical case for $\boldsymbol{N=4}\,$: Proof of Theorem \ref{Thm-C-4}}\label{Section_crit_N4}

The main strategy is similar to the one used to deal with the critical case in dimensions $N\geq5$. First, we note that the existence of a radial ground state solution for the problem \eqref{eq} in the whole space under conditions ($f_2$),($f_4$),($f_6$),($f_7$) has been proved in \cite{Sani_critical}. In fact, the author proves the existence of a radial mountain-pass solution but, provided ($f_5$) holds, it is standard to prove that the radial mountain-pass level coincides with the least energy level on the radial Nehari manifold. Note also that in \cite{Sani_critical} the author deal with the case $\gamma=0$, but it is not difficult to retrace the arguments therein in the general case $\gamma\geq0$.

As in Section \ref{Proof-S}, we need to focus now on the approximating Dirichlet problems, aiming at proving existence of radial ground state solutions and some compactness properties.
\vskip0.2truecm

We recall a technical estimate from \cite{Sani_subcritical}.
\begin{lem}\label{Lem_Sani-C}
	Let $\alpha>0$ and $q\geq2$. If $M>0$ and $\alpha M<32\pi^2$ then there exists a constant $C(\alpha,M,q)>0$ such that the inequality
	\begin{equation*}
		\int\left(\e^{\alpha u^2}-1\right)|u|^q\dd x\leq C(\alpha,M,q)\|u\|_{H^2(\R^4)}^q
	\end{equation*}
	holds for any $u\in H^2(\R^4)$ with $\|u\|_{H^2(\R^4)}\leq M$.
\end{lem}

For $R>0$ the radial mountain-pass level $c_R$ is defined as in \eqref{MPlevel}.
\begin{lem}\label{MP_lem-C}
	For all $R\geq1$ there exists a mountain-pass solution $u_R\in H^2_{0,\rad}(B_R)$ of problem \eqref{eq_R} at level $c_R$. Moreover, the map $R\mapsto c_R$ is monotone decreasing, and there exists $\delta>0$ such that $c_R>\delta$ for all $R\geq1$.
\end{lem}
\begin{proof}
	The only difference in the proof with respect to the subcritical case (Lemma \ref{MPunif}) relies in the estimate from below of the functional $J_R$ restricted to sufficiently small balls around $0\in H^2_{0,\rad}(\R^4)$. It is hence sufficient to estimate $\int_{B_R}F(u)$ by means of Lemma \ref{Lem_Sani-C}. In fact, for $\iii u_R=\rho>0$ and choosing $q>2$, $\alpha>\alpha_0$, and $\varepsilon>0$, by \eqref{f_above_N=4} one has
	$$\int_{B_R}F(u)\dd x\leq\varepsilon\|\tu\|_2^2+C(\alpha,q,\varepsilon)\int|\tu|^q\left(\e^{\alpha\tu^2}-1\right)\dd x.$$
	If $\rho<\pi\sqrt{\frac{32}{\alpha C_{eq}}}$ one may apply Lemma \ref{Lem_Sani-C} and obtain
	$$\int|\tu|^q\left(\e^{\alpha\tu^2}-1\right)\dd x\leq C\|\tu\|_{H^2(\R^4)}^q\leq CC_{eq}^q\iii\tu^q,$$
	so that
	$$J_R(u)=\frac12\iii u_R^2-\int_{B_R}F(u)\dd x\geq\left(\frac12-C\varepsilon\right)\iii\tu^2-\widetilde C(\varepsilon)\iii\tu^q>\delta$$
	for some $\delta>0$, provided $\varepsilon$ and $\iii\tu=\rho$ are small enough.
\end{proof}

The next aim is to obtain compactness for the solution family $(u_R)_R$ and convergence to a solution for the equation in the whole space. Some steps of the proof are reminiscent of \cite[Sections 4 and 6]{Sani_critical} and we will use the so-called \textit{radial lemma}, see \cite[Lemma A.II]{BL}.
\begin{lem}\label{radial_lemma}
	For any $u\in H^1_{\rad}(\R^4)$ one has
	\begin{equation}
		|u(x)|\leq\frac1{\sqrt{2\pi^2}}\frac{\|u\|_{H^1(\R^4)}}{|x|^\frac32}\qquad\mbox{a.e. in}\ \,\R^4.
	\end{equation}
\end{lem}

\begin{lem}\label{exists_uinfty-C}
	The (extended) family of mountain-pass solutions $(\tu_R)_R$ is uniformly bounded in $H^2_{0,\rad}(\R^4)$. Moreover, there exists $u_\infty\in H^2_{\rad}(\R^4)$ such that, up to a subsequence, $\tu_R\rightharpoonup u_\infty$ and $\uinfty$ is a radial nonnegative weak solution of \eqref{eq}.
\end{lem}
\begin{proof}
	As in the subcritical case (Proposition \ref{exists_uinfty}), using the uniform boundedness of $(c_R)_R$ we infer that the solution family $(\tu_R)_R$ is uniformly bounded in $H^2(\R^4)$. Hence, there exists a subsequence denoted by $(\tu_k)_k$ and $\uinfty\in H^2_{\rad}(\R^4)$ such that $\tu_k\rightharpoonup\uinfty$ in $H^2(\R^4)$ and, by compact embedding, $\tu_k\to\uinfty$ in $L^p(\R^4)$ for $p\in(2,+\infty)$ and a.e. in $\R^4$. Let now $\varphi\in C^\infty_0(\R^4)$; for $k$ large enough $B_{R_K}(0)\supset\supp\,\varphi$ and hence \eqref{conv_1-S} holds. In order to show that $u_\infty$ is a weak solution for \eqref{eq} it is sufficient to prove that
	\begin{equation}\label{f_conv_loc}
		\int f(u_k)\varphi\to\int f(\uinfty)\varphi
	\end{equation}
	as $k\to+\infty$. Note also that as in the subcritical case, by Lemma \ref{Pos_CT} one has $u_k>0$ for $k$ large enough. First, we have that $\int f(u_k)\varphi$ is bounded uniformly in $k$ by means of \eqref{conv_1-S} and the fact that $u_k$ solves \eqref{eq_R} in $B_{R_k}$. Since $\uinfty\in H^2(\R^4)$ and $J$ is well-defined on that space, we have also $\int f(\uinfty)\varphi<+\infty$. Finally, since $u_k$ is a critical point of $J_{R_k}$, one has
	$$\int |f(u_k)u_k\varphi|=\int_{\supp\varphi}f(u_k)u_k|\varphi|\leq\|\varphi\|_\infty\int_{B_{R_k}}f(u_k)u_k=\|\varphi\|_\infty\iii{u_k}_{R_k}^2\leq C.$$
	Hence, by \cite[Lemma 2.1]{dFMR} we get $\int f(\tu_k)\varphi\to\int f(\uinfty)\varphi$ and conclude that $\uinfty$ is a nonnegative radial solution of \eqref{eq}.
\end{proof}

Next we prove that $\uinfty$ is nontrivial, by showing that its energy is positive
\begin{lem}\label{Lem_F_conv}
	We have
	\begin{equation}\label{F_conv}
		\int F(\tu_k)\dd x\to\int F(\uinfty)\dd x.
	\end{equation}
\end{lem}
\begin{proof}
	Using ($f_6$), as in \cite[(31)-(32)]{Sani_critical} one proves that for all $\bR>0$
	\begin{equation}\label{F_loc-C}
		\int_{B_\bR}F(\tu_k)\dd x\to\int_{B_\bR}F(\uinfty)\dd x.
	\end{equation}
	Moreover, for $\varepsilon>0$ and $\alpha>\alpha_0$, \eqref{F-C-above} implies
	\begin{equation}\label{F_ext-1-C}
		\int_{\R^4\setminus B_\bR}F(\tu_k)\dd x\leq\varepsilon\|\tu_k\|_{L^2(\R^4\setminus B_\bR)}^2+C(\alpha,\varepsilon)\int_{\R^4\setminus B_\bR}\tu_k\left(\e^{\alpha \tu_k^2}-1\right)\dd x.
	\end{equation}
	We estimate the second term by means of the power series expansion of the exponential function, together with Lemma \ref{radial_lemma}:
	\begin{equation*}
		\begin{split}
			\int_{\R^4\setminus B_\bR}\tu_k\left(\e^{\alpha\tu_k^2}-1\right)\dd x&=\sum_{j=1}^{+\infty}\frac{\alpha^j}{j!}\int_{\R^4\setminus B_\bR}\tu_k^{2j+1}\dd x\\
			&\leq2\pi^2\sum_{j=1}^{+\infty}\frac{\alpha^j}{j!}\left(\frac1{2\pi^2}\right)^{j+\frac12}\|\tu_k\|_{H^1(\R^4)}^{2j+1}\frac{\bR^{\frac52-3j}}{3j-\frac52}\\
			&\leq\frac{\sqrt{2\pi^2}}{\sqrt\bR}\|\tu_k\|_{H^1(\R^4)}\sum_{j=1}^{+\infty}\frac1{j!}\left(\frac{\alpha\|\tu_k\|_{H^1(\R^4)}^2}{2\pi^2}\right)^j\\
			&\leq\frac{\sqrt{2\pi^2}}{\sqrt\bR}\|\tu_k\|_{H^1(\R^4)}\e^{\frac\alpha{2\pi^2}\|\tu_k\|_{H^1(\R^4)}}.
		\end{split}
	\end{equation*}
	Therefore, from \eqref{F_ext-1-C}, since $\|\tu_k\|^2_{H^1(\R^4)}\leq C_{eq}\iii{\tu_k}^2\leq C$, one infers
	\begin{equation}\label{F_ext-C}
		\int_{\R^4\setminus B_\bR}F(\tu_k)\dd x\leq C\varepsilon+\frac{C(\alpha,\varepsilon)}\bR.
	\end{equation}
	Analogously one also obtains
	\begin{equation}\label{F_ext-infty-C}
		\int_{\R^4\setminus B_\bR}F(\uinfty)\dd x\leq C\varepsilon+\frac{C(\alpha,\varepsilon)}\bR.
	\end{equation}
	Since $\varepsilon>0$ and $\bR$ are arbitrary, \eqref{F_conv} holds combining \eqref{F_loc-C}-\eqref{F_ext-infty-C}.
\end{proof}

\begin{lem}\label{Lem_ufu0}
	The weak solution $\uinfty$ is nontrivial and nonnegative in $\R^4$.
\end{lem}
\begin{proof}
	Note that Supposing by contradiction that $\uinfty\equiv0$, we claim that
	\begin{equation}\label{ufu0_claim}
		\int f(\tu_k)\tu_k\to0
	\end{equation}
	as $k\to+\infty$. Indeed, if so, since $u_k$ solves \eqref{eq_R} in $B_{R_k}$ and $J_{R_k}(u_k)=c_{R_k}$, one would establish a contradiction with
	\begin{equation}\label{ufu0_contradiction}
		\int f(\tu_k)\tu_k=\iii{\tu_k}^2=2c_{R_k}+\int F(\tu_k)\to2c_\infty\geq\delta,
	\end{equation}
	by Lemma \ref{Lem_F_conv}, where $\delta>0$ is given by Lemma \ref{MP_lem-C} and $c_\infty:=\displaystyle{\lim_{R\to+\infty}c_R}$.
	
	In order to prove the claim \eqref{ufu0_claim}, let $q\geq1$, $\alpha>\alpha_0$ and $\varepsilon>0$. Then by \eqref{f-C-above} we have
	\begin{equation}
		\begin{split}
			\int f(\tu_k)\tu_k&\leq\varepsilon\|\tu_k\|_2^2+C(\alpha,q,\varepsilon)\int|\tu_k|^q\left(\e^{\alpha\tu_k^2}-1\right)\dd x\\
			&\leq\frac\varepsilon\gamma_0\iii{\tu_k}^2+C(\alpha,q,\varepsilon)\|\tu_k\|_{qp'}^q\left(\int\left(\e^{\alpha p\tu_k^2}-1\right)\dd x\right)^\frac1p,
		\end{split}
	\end{equation}
	having applied H\"older inequality with $p>1$, together with Lemma \ref{estimate_Sani}. Since $\tu_k\to0$ in $L^{qp'}(\R^4)$ by compact embedding, provided one chooses $q>2$, we are left to prove that
	\begin{equation*}
		\sup_k\int\left(\e^{\alpha p\tu_k^2}-1\right)\dd x<+\infty
	\end{equation*}
	for suitable choices of $\alpha>\alpha_0$ and $p>1$. By means of assumptions ($f_6$)-($f_7$), using the same argument of \cite[Lemma 8]{Sani_critical}, we infer that there exists $R_0>0$ such that for all $R\geq R_0$ one may uniformly bound the mountain-pass level in a sharp way and obtain
	$$c_R<\frac{16\pi^2}{\alpha_0}\,.$$
	The constant $R_0$ is related to the support of the Moser sequence used in the proof, which is uniformly bounded with respect to $R$. Since the map $R\mapsto c_R$ is decreasing, we also infer 
	\begin{equation}\label{c_mp_Sani}
		c_\infty<\frac{16\pi^2}{\alpha_0}
	\end{equation}
	and, by Lemma \ref{Lem_F_conv},
	\begin{equation*}
		\iii{\tu_k}^2=2c_{R_k}+2\int F(\tu_k)\to2c_\infty<\frac{32\pi^2}{\alpha_0}.
	\end{equation*}
	Therefore, there exist constants $\nu,\sigma>0$ and $k_0\in\N$ such that
	\begin{equation*}
		{\|\Delta\tu_k\|_2^2\leq2c_\infty+\frac\nu2\qquad\mbox{and}\qquad2c_\infty+\nu<\frac{32\pi^2}{\alpha_0}(1-\sigma)}
	\end{equation*}
	for all $k\geq k_0$. Since $(\tu_k)_k$ is uniformly bounded in $H^2(\R^4)$, then there exists a constant $\tau>0$ small such that
	\begin{equation*}
		2\tau\|\nabla\tu_k\|_2^2+\tau^2\|\tu_k\|_2^2<{\frac\nu2},
	\end{equation*}
	and hence
	\begin{equation*}
		\|\tu_k\|_{H^2,\tau}^2<\frac{32\pi^2}{\alpha_0}(1-\sigma)
	\end{equation*}
	for all $k\geq k_0$. Therefore, it is possible to choose $\alpha>\alpha_0$ close to $\alpha_0$ and $p>1$ close to $1$, so that
	$$\alpha p\|\tu_k\|_{H^2,\tau}^2\leq32\pi^2\frac\alpha{\alpha_0}p(1-\sigma)\leq32\pi^2.$$
	Hence, using Ruf-Sani's version \eqref{Adams_RS-tau} of the Adams' inequality in the whole space, we deduce
	\begin{equation}
		\sup_k\int\left(\e^{\alpha p\tu_k^2}-1\right)\dd x=\sup_k\int\left(\e^{32\pi^2\left(\tfrac{\tu_k}{\|\tu_k\|_{H^2,\tau}}\right)^2}-1\right)\dd x<+\infty
	\end{equation}
	The claim \eqref{ufu0_claim} is then proved and therefore $\uinfty\not\equiv0$.
\end{proof}

\begin{proof}[Proof of Theorem \ref{Thm-C-4}]
	In Lemma \ref{exists_uinfty-C} we already obtained a radial weak solution $\uinfty$ of \eqref{eq} as limit of the solution sequence $(\tu_k)_k$, and proved that $\uinfty$ is nontrivial in Lemma \ref{Lem_ufu0}. We are left to show that $\uinfty$ is a radial \textit{ground state} solution to \eqref{eq}. To see this, we first show that
	\begin{equation}\label{conv_ufu}
		\int f(\tu_k)\tu_k\to\int f(\uinfty)\uinfty.
	\end{equation}
	Indeed, if so, then $\iii{\tu_k}\to\iii{\uinfty}$ holds using the fact that $u_k$ and $\uinfty$ are solutions, and this, together with the weak convergence in $H^2(\R^4)$, leads to the strong convergence $\tu_k\to\uinfty$ in $H^2(\R^4)$. Having this in hand, and since the existence of a radial ground state solution for equation \eqref{eq} is known as described above, one can proceed with the same argument used in the proof of Theorem \ref{Thm2} and obtain the existence of a nonnegative radial ground state solution, positive on $\R^N\setminus\{0\}$.
	
	In order to prove \eqref{conv_ufu}, let us split
	\begin{equation}\label{conv_ufu_split}
		\left|\int f(\tu_k)\tu_k-\int f(\uinfty)\uinfty\right|\leq\int f(\tu_k)|\tu_k-\uinfty|+\int\left|f(\tu_k)-f(\uinfty)\right|\uinfty
	\end{equation}
	and estimate the two terms separately. Concerning the first term, for $\alpha>\alpha_0$, $q\geq2$ and $\varepsilon>0$ one may estimate
	\begin{equation}\label{ukfuk-ufu}
		\begin{split}
			\int& f(\tu_k)|\tu_k-\uinfty|\leq\varepsilon\left(\|\tu_k\|_2^2+\|\uinfty\|_2^2\right)+C(\alpha,q,\varepsilon)\int|\tu_k|^{q-1}|\tu_k-\uinfty|\left(\e^{\alpha\tu_k^2}-1\right)\dd x\\
			&\leq\varepsilon\left(\frac1\gamma_0\iii{\tu_k}^2+\|\uinfty\|_2^2\right)+C(\alpha,q,\varepsilon)\|\tu_k\|_{(q-1)p'\sigma'}^{q-1}\|\tu_k-\uinfty\|_{p'\sigma}\left(\int\left(\e^{\alpha p\tu_k^2}-1\right)\dd x\right)^\frac1p,
		\end{split}
	\end{equation}
	having applied H\"older inequality twice with $p,\sigma>1$ and Lemma \ref{estimate_Sani}. Note that here
	\begin{equation}\label{conv_|||u_k|||}
		\iii{\tu_k}^2=2c_{R_k}+2\int F(\tu_k)\to2c_\infty+2\int F(\uinfty),
	\end{equation}
	by Lemma \ref{Lem_F_conv}, where the second term is positive.
	Defining
	\begin{equation*}
		v_k:=\frac{\tu_k}{\iii{\tu_k}}\qquad\mbox{and}\qquad v_\infty:=\frac{\uinfty}{\sqrt{2c_\infty+2\int F(\uinfty)}}.
	\end{equation*}
	we have $\iii{v_k}=1$ for all $k\in\N$ and $\iii{v_\infty}\leq1$. If $\iii{v_\infty}=1$, then $J(\uinfty)=c_\infty$, which in turns yields $\iii{\tu_k}\to\iii{u_\infty}$ by Lemma \ref{Lem_F_conv}, and the claim \eqref{conv_ufu} is proved. Let us thus suppose that $\iii{v_\infty}<1$, which implies that $J(\uinfty)<c_\infty$. By \cite[Lemma 3]{Sani_critical} one infers
	\begin{equation}\label{Lemma_3_Sani}
		\sup_k\int\left(\e^{rv_k^2}-1\right)\dd x<+\infty\quad\mbox{for all}\ \ r\in\left(0,\frac{32\pi^2}{1-\iii{v_\infty}^2}\right).
	\end{equation}
	Since
	\begin{equation*}
		\sup_k\int\left(\e^{\alpha p\tu_k^2}-1\right)\dd x=\sup_k\int\left(\e^{\alpha p\iii{\tu_k}^2v_k^2}-1\right)\dd x,
	\end{equation*}
	we need to show that we may take $\alpha>\alpha_0$ and $p>1$ such that
	\begin{equation*}
		\alpha p\iii{\tu_k}^2<\frac{32\pi^2}{1-\iii{v_\infty}^2}=32\pi^2\frac{c_\infty+\int F(\uinfty)}{c_\infty-J(\uinfty)}.
	\end{equation*}
	By \eqref{conv_|||u_k|||}, it is sufficient to prove
	\begin{equation}\label{alphabeta}
		\alpha p<\frac{16\pi^2}{c_\infty-J(\uinfty)}.
	\end{equation}
	Recalling that $c_\infty\leq c_R<\frac{16\pi^2}{\alpha_0}$ and $J(\uinfty)>0$, then
	\begin{equation*}
		\alpha_0<\frac{16\pi^2}{c_\infty}<\frac{16\pi^2}{c_\infty-J(\uinfty)},
	\end{equation*}
	then for $\alpha>\alpha_0$ close to $\alpha_0$, $p>1$ close to $1$, one obtains \eqref{alphabeta} and thus by \eqref{Lemma_3_Sani} we conclude
	\begin{equation*}
		\sup_k\int\left(\e^{\alpha p\tu_k^2}-1\right)\dd x=\sup_k\int\left(\e^{\alpha p\iii{\tu_k}^2v_k^2}-1\right)\dd x<+\infty.
	\end{equation*}

	Therefore, it is sufficient to choose $q>3$ to have $\|\tu_k\|_{(q-1)p'\sigma'}\leq\iii{\tu_k}\leq C$ and $\sigma>1$ such that $p'\sigma>2$ and thus $\|\tu_k-\uinfty\|_{p'\sigma}\to0$. Hence by \eqref{ukfuk-ufu} the first term in \eqref{conv_ufu_split} converges to $0$.
	
	To deal with the second term in \eqref{conv_ufu_split} one may argue as in Lemma \ref{Lem_F_conv}. Indeed, since $\uinfty\in L^\infty_{loc}(\R^4)$ by elliptic regularity, by \eqref{f_conv_loc} it is immediate to conclude that
	\begin{equation}\label{ufu_conv_loc}
		\int_{B_\bR} f(\tu_k)\uinfty\to\int_{B_\bR} f(\uinfty)\uinfty
	\end{equation}
	for all $\bR$ sufficiently large. Moreover, for all $\alpha>\alpha_0$ one has
	\begin{equation}\label{ufu_ext-1-C}
		\int_{\R^4\setminus B_\bR}f(\tu_k)\uinfty\dd x\leq\int_{\R^4\setminus B_{\bR}}\tu_k\uinfty+C(\alpha)\int_{\R^4\setminus B_\bR}\uinfty\left(\e^{\alpha\tu_k^2}-1\right)\dd x.
	\end{equation}
	Since for any $\varepsilon>0$ there exists $\bR$ large enough so that
	\begin{equation*}
		\int_{\R^4\setminus B_{\bR}}\tu_k\uinfty\leq\iii{\tu_k}_2\left(\int_{\R^4\setminus B_{\bR}}|\uinfty|^2\right)^\frac12\leq C\varepsilon
	\end{equation*}
	and
	\begin{equation*}
		\int_{\R^4\setminus B_\bR}\uinfty\left(\e^{\alpha\tu_k^2}-1\right)\dd x\leq\sqrt\frac{2\pi^2}\bR\|\uinfty\|_{H^1(\R^4)}\e^{\frac\alpha{2\pi^2}\|\tu_k\|_{H^1(\R^4)}}\leq\frac C\bR,
	\end{equation*}
	because $\|\tu_k\|_{H^1(\R^4)}$ is uniformly bounded in $k$, from \eqref{ufu_ext-1-C} we infer
	\begin{equation}\label{ufu_ext_k}
		\int_{\R^4\setminus B_\bR}f(\tu_k)\uinfty\dd x\leq C\varepsilon+\frac C\bR.
	\end{equation}
	In the same way one also infers that
	\begin{equation}\label{ufu_ext_infty}
		\int_{\R^4\setminus B_\bR}f(\uinfty)\uinfty\dd x\leq C\varepsilon+\frac C\bR.
	\end{equation}
	Combining \eqref{ufu_conv_loc}, \eqref{ufu_ext_k}, and \eqref{ufu_ext_infty}, one eventually deduces \eqref{conv_ufu}, and in turn the claimed convergence $\tu_k\to\uinfty$ in $H^2(\R^4)$. The comparison with a radial ground state, as in the proof of Theorem \ref{Thm2}, gives $J(\uinfty)=c_{min}$. Lemma \ref{lem:no-zero-spheres} then implies $\uinfty(x)>0$ for every $x\ne0$. If $\gamma^2\geq4\gamma_0$, Proposition \ref{prop:factorization-positive} gives $\uinfty(0)>0$ as well.
\end{proof}

\section*{Acknowledgments} G. Romani is partially supported by the INdAM-GNAMPA 2026 project \textit{Structural degeneracy and criticality in (sub)elliptic PDEs} (E53C25002010001). Z. Liu is supported by  the Fundamental Research Funds for the Central Universities, China University of Geosciences (Nos. CUG2106211; CUGST2), and Guangdong Basic and Applied Basic Research Foundation (Nos. 2023A1515011679; 2024A1515012704).

\end{document}